\documentclass[12pt,a4paper]{article}

\usepackage{hyperref}

\usepackage[english]{babel}
\usepackage{fancyhdr}
\usepackage{fullpage}
\usepackage[dvips]{graphicx}
\usepackage[T1]{fontenc}
\usepackage{fouriernc}
\usepackage[latin1]{inputenc}
\usepackage{amsmath}
\usepackage{amsfonts}
\usepackage{amssymb}
\usepackage{mathrsfs}
\usepackage{amsthm}
\usepackage{latexsym}
\usepackage{array}
\usepackage{mathrsfs}
\usepackage{amsxtra}
\usepackage{amscd}
\usepackage{mathtools}
\usepackage{verbatim}
\usepackage{enumerate}
\usepackage[title]{appendix}
\usepackage[usenames]{color}
\definecolor{beige}{rgb}{0.96, 0.96, 0.86}
\definecolor{airforceblue}{rgb}{0.36, 0.54, 0.66}
\definecolor{antiquefuchsia}{rgb}{0.57, 0.36, 0.51}
\definecolor{awesome}{rgb}{1.0, 0.13, 0.32}

\usepackage{tikz} 
\usetikzlibrary{arrows,shapes,matrix}

\usepackage{xcolor}
\hypersetup{
    colorlinks,
    linkcolor={green!60!blue},
    citecolor={yellow!30!red},
    urlcolor={green!60!blue}
}

\newcommand*{\Cb}[1]{C([0,T],#1)}

\newcommand*{\medcap}{\mathbin{\scalebox{1.5}{\ensuremath{\cap}}}}

\newcommand*{\medcup}{\mathbin{\scalebox{1.5}{\ensuremath{\cup}}}}

\numberwithin{equation}{section}

\theoremstyle{plain}
\newtheorem{theorem}{Theorem}[section]

\newtheorem{proposition}[theorem]{Proposition}
\newtheorem{lemma}[theorem]{Lemma}

\newtheorem{assumption}[theorem]{Assumption}

\theoremstyle{definition}
\newtheorem{remark}[theorem]{Remark}
\newtheorem{example}[theorem]{Example}

\renewcommand{\epsilon}{\varepsilon}
\renewcommand{\phi}{\varphi}

\def\qed{{\hfill\hbox{\enspace${ \blacksquare}$}} \smallskip}

\newcommand{\e}{e}
\newcommand{\leb}{L}

  \def\Swiech
{{\accent1 S}wie{\hbox{\kern -0.31em\lower
0.77ex \hbox{$\textfont1=\scriptfont1 \lhook$}}}ch}

\def\SWIECH
{{\accent1 S}WIE{\hbox{\kern -0.21em\lower
0.77ex\hbox{$\textfont1=\scriptfont1
\lhook$}}}CH}

\usepackage{dsfont} 

\def\b*{\begin{eqnarray*}}
\def\e*{\end{eqnarray*}}

\def \0{\mathbf{0}}

\theoremstyle{plain}

\theoremstyle{definition}

\newcommand*{\blp}{(}
\newcommand*{\brp}{)}

\theoremstyle{plain}

\theoremstyle{definition}

 \newtheorem*{notation}{Notation}

 \theoremstyle{definition}

\def\<{\left\langle }
\def\>{\right\rangle }

\title{A note on stochastic Fubini's theorem and stochastic convolution}
\date{}
\author{Mauro Rosestolato\thanks{CMAP, \'Ecole Polytechnique, Paris, France,
e-mail: \texttt{mauro.rosestolato@polytechnique.edu}. 
This research has been
partially supported by
the ERC
321111 Rofirm.
}
}

\begin{document}

\maketitle

\begin{abstract} 
We provide a version of the stochastic Fubini's theorem which does not depend on the particular stochastic integrator chosen 
as far as
the 
stochastic integration 
is built as a continuous linear
operator
from an $L^p$ space of 
Banach space-valued
 processes (the stochastically integrable processes) to an $L^p$ space of Banach space-valued paths (the integrated processes).
Then, for integrators on a Hilbert space $H$,  we  consider stochastic convolutions with respect to a strongly continuous map $R\colon (0,T]\rightarrow L(H)$, not necessarily a semigroup.
We prove existence of predictable versions of stochastic convolutions and we characterize the measurability needed by operator-valued processes in order to be convoluted with $R$.
Finally, when $R$ is a $C_0$-semigroup and the stochastic integral provides continuous paths, we show existence of a continuous version of the convolution, by adapting the factorization method to the present setting.
\end{abstract}

\vspace{10pt}

\vspace{10pt} 

\section{Introduction}




In this note
we prove 
a stochastic Fubini's theorem and apply it to obtain
existence of predictable/continuous versions of stochastic convolutions.
We do not choose any particular stochastic integrator.
We look at the stochastic integration simply as a  linear and continuous operator $ \mathfrak {L}$ 
from an $L^p$ space of Banach space-valued processes, the stochastically integrable processes, to another $L^p$ space, containing
functions whose values are 
 the paths of the stochastic integrals.
The paths do not need to be continuous.
Within this setting,
the continuity assumption on 
$ \mathfrak {L}$ plays the  role of It\=o's isometry
or of the Burkholder-Davis-Gundy inequality
 in the standard construction of stochastic integrals with respect to square integrable continuous martingales.

For such an operator $ \mathfrak L$, we  prove
 the stochastic Fubini's theorem (Theorem~\ref{2015-09-21:01}).
The result can be applied e.g.\ to stochastic integration in infinite dimensional spaces with respect to $L^p$-integrable martingales (\cite[Ch.\ 8]{Peszat2007}) or more general martingale-valued measures (for the finite dimensional case, see e.g.\ \cite[Ch.\ 4]{Applebaum2009}), generalizing standard results as 
\cite[Theorem 4.33]{DaPrato2014}, \cite[Theorem 2.8]{Gawarecki2011}, \cite[Theorem 8.14]{Peszat2007}.

Secondly,  we particularize the study to the case in which $ \mathfrak  L$ is defined on a space of
  $L_2(U,H)$-valued processes,
 where $U,H$ are separable Hilbert spaces
 and $L_2(U,H)$ is the vector space of Hilbert-Schmidt operators from $U$ into $H$.
Denote $ \mathfrak L$ by $ \mathfrak I$,
in this particular case.
For a strongly continuous map $R\colon (0,T]\rightarrow L(H)$ and for a process $\Phi\colon \Omega\times [0,T]\rightarrow L(U,H)$ such that the composition 
$\mathbf{1}_{(0,t]}(\cdot)R(t-\cdot)\Phi$ 
 belongs to the domain of $ \mathfrak I$, we consider the convolution process
\begin{equation}
  \label{eq:2017-06-17:00}
   (\mathfrak I(\mathbf{1}_{[0,t)}(\cdot)R(t-\cdot)\Phi))_t\qquad t\in[0,T].
 \end{equation}
 By using the stochastic Fubini's theorem, we show that \eqref{eq:2017-06-17:00} admits a jointly measurable version
(Theorem~\ref{2017-05-07:11}).
The joint measurability of the stochastic convolution is of interest e.g.\ when its paths must be integrated, as it happens in the factorization formula (\cite[Theorem 5.10]{DaPrato2014}).
We also
provide a characterisation of the measurability needed by functions $\Phi\colon\Omega\times [0,T]\rightarrow L(U,H)$ in order
that $\mathbf{1}_{(0,t]}(\cdot)R(t-\cdot)\Phi$ has the necessary measurability required by the operator $ \mathfrak I$ (Theorem~\ref{2017-05-07:12}).
This measurability result turns out to be useful e.g.\ 
in order to understand
what are the most general measurability conditions for  coefficients of stochastic differential equations in Hilbert spaces for which mild solutions are considered.

Finally, in case $ \mathfrak I$ takes values in a space of processes with continuous paths and $R=S$ is a $C_0$-semigroups, by adapting
the factorization method to the present setting,
 we show that 
\eqref{eq:2017-06-17:00} admits a continuous version (Theorem~\ref{2016-01-17:06}).

\section{Stochastic Fubini's theorem}
\label{2016-01-19:05}

  Throughout this  section,
   $(G,\mathcal{G},\mu)$
and $(D_2,\mathcal{D}_2,\nu_2)$
 are  positive finite measure spaces, $(D_1,\mathcal{D}_1)$ is a measurable space, and $\nu_1$ is a kernel from $D_2$ to $D_1$, i.e.\ 
$$
\nu_1\colon \mathcal{D}_1\times D_2\rightarrow \mathbb{R}^+
$$
is such that
\begin{enumerate}[(i)]
\item $\nu_1(A,\cdot)$ is $\mathcal{D}_2$-measurable, for all $A\in \mathcal{D}_1$;
\item $\nu_1(\cdot,x)$ is a positive measure, for all $x\in D_2$.
\end{enumerate}
We assume that
$$
C\coloneqq \int_{D_2} \nu_1(D_1,x)\nu_2(dx)<\infty.
$$
Let  $D\coloneqq D_1\times D_2$.
On $(D,\mathcal{D}_1 \otimes \mathcal{D}_2)$, we define the meaure $\nu$ by
$$
\nu(A)\coloneqq \int_{D_2} 
\left( 
  \int_{D_1}\mathbf{1}_{A}(x_1,x_2)
  \nu_1(dx_1,x_2) 
\right) \nu_2(dx_2),\qquad \forall A\in \mathcal{D}_1 \otimes \mathcal{D}_2.
$$
Notice that $\nu(D)=C$ is finite.

Let $\mathcal{D}$  be 
a given sub-$\sigma$-algebra 
of $\mathcal{D}_1 \otimes \mathcal{D}_2$.
 When we consider   measurability or integrability with respect to $G$ (resp.\ $D_1$, $D_2$, $D_1\times D_2$, $D$), we always mean it with respect to the space $ ( G,\mathcal{G},\mu ) $ (resp.\ $ ( D_1,\mathcal{D}_1,\nu_1 ) $,
$( D_2,\mathcal{D}_2,\nu_2 )$, 
$( D,\mathcal{D},\nu )$).
According to that, if we write, for example
 $L^1(D,V)$,
for some Banach space $V$,
 we mean $L^1((D,\mathcal{D},\nu),V)$, and similarly for other spaces of integrable functions on $G$, $D_1$, $D_2$, $D$.

Let $E$ be a given Banach space. For
 $p,q\in [1,\infty)$,
we denote by $L^{p,q}_\mathcal{D}(E)$  the space of measurable functions $f\colon (D,\mathcal{D})\rightarrow E$ such that
\begin{enumerate}[(i)]
\item there exists $N\in \mathcal{D}$ such that $\nu(N)=0$ and $f(G\setminus N)$ is separable;
\item 
the following integrability condition holds:
$$
|f|_{p,q}\coloneqq  \left( \int_{D_2} \left( \int_{D_1}|f(x,y)|_E^pd\nu_1(dx,y) \right)^{q/p}  \nu_2(dy) \right) ^{1/q}<\infty.
$$
\end{enumerate}

It is not difficult to see that $(L^{p,q}_{\mathcal{D}}(E),|\cdot|_{L^{p,q}_{\mathcal{D}}(E)})$ is a Banach space, with the usual identification $f=g$ if and only if $f=g$ $\nu$-a.e..
Indeed, if $\{f_n\}_{n\in \mathbb{N}}$ is Cauchy in $L^{p,q}_{\mathcal{D}}(E)$, then it is Cauchy also in $L^1((D,\mathcal{D},\nu),E)$. 
Passing to a subsequence if necessary, we may assume that $f_n\rightarrow f$ $\nu$-a.e., for some $f\in L^1((D,\mathcal{D},\nu),E)$.
Now Fatou's lemma gives $f\in L^{p,q}_{\mathcal{D}}(E)$ and $f_n\rightarrow f$ in $L^{p,q}_{\mathcal{D}}(E)$.

Finally, we use the short notation
$L^1(D\times G,E)$ for the space
$$
L^1((D\times G,\mathcal{D} \otimes \mathcal{G},\nu \otimes  \mu),E).
$$

\bigskip
We will prove
the stochastic Fubini's theorem
first for simple functions and then
for the general case 
through
 approximation.
We need the following
  preparatory  lemma.

\begin{lemma}\label{2015-08-15:00}
 Let $p,q\in [1,\infty)$  and  
$f\in L^1(G,L^{p,q}_{\mathcal{D}}(E))$.
\noindent If $q>1$, assume that
\begin{equation}
  \label{eq:2017-05-06:00}
  C(p,q)\coloneqq
 \left( 
 \int_{D_2}  \left( \nu_1(D_1,x) \right) ^{\frac{q(p-1)}{p(q-1)}}
\nu_2(dx)
 \right) ^{\frac{q-1}{q}}
<\infty.
\end{equation}
If $q=1$ and $p>1$, assume that
\begin{equation}
  \label{eq:2017-05-06:01}
  C(p,1)\coloneqq
\sup_{x\in D_2}
 \left(  \nu_1(D_1,x) \right) ^{\frac{p-1}{p}}
<\infty.
\end{equation}
Define $C(1,1)\coloneqq 1$.
Then there exist measurable functions
\begin{gather}
   \tilde f\colon   \left(  D\times G,
 \mathcal{D}  \otimes  \mathcal{G}\right)   \rightarrow E\label{2015-09-14:05}\\ 
   \tilde f_n\colon  \left(  D\times G,
 \mathcal{D} \otimes  \mathcal{G} \right) \rightarrow E,\ n\in \mathbb{N}\label{2015-09-14:06} 
 \end{gather}
such that
\begin{gather}
  \tilde  f(\cdot,y)\in L^{p,q}_\mathcal{D}(E),\ \forall y\in G,\label{2015-08-14:21}\\
  G\rightarrow L^{p,q}_\mathcal{D}(E),\ y\mapsto \tilde f(\cdot,y) \  \mbox{is measurable}\label{2015-08-14:22}\\ 
  \tilde f(\cdot,y)=f(y) \ \mbox{in }L^{p,q}_\mathcal{D}
(E) \ \mu\mbox{-a.e.\ }y\in G,\label{2015-08-14:15}
\end{gather}
\begin{gather}
  \tilde f_n(\cdot,y)\in L^{p,q}_\mathcal{D}
(E),\ 
\forall y\in G,\ \forall n\in \mathbb{N},
\label{2015-09-14:03}\\
  G\rightarrow L^{p,q}_\mathcal{D}(E),\ y\mapsto \tilde f_n(\cdot,y)\mbox{ is a simple function, }
\forall  n\in \mathbb{N},
\label{2015-09-14:04}\\
\hskip-1cm\lim_{n\rightarrow \infty}
\int_G \left( 
\int_{D_2}
 \left( 
\int_{D_1}
|\tilde f_n((x_1,x_2),y)
-
\tilde f((x_1,x_2),y)
|_E^p
\nu_1(dx_1,x_2)
 \right) ^{q/p}
\nu_2(dx_2)
 \right)^{1/q}
\mu(dy)=0
  \label{2015-08-15:06}
\end{gather}
\begin{gather}
  \tilde f(x,\cdot)\in L^1(G,E),\ \forall
  x\in D,\label{2015-08-14:16}\\
  D\rightarrow L^1(G,E),\ x\mapsto \tilde f(x,\cdot),
\ \mbox{belongs to\ } L^{p,q}(D,L^1(G,E))
\label{2015-09-14:10}
\end{gather}
\begin{gather}
  \tilde f_n(x,\cdot)\in L^1(G,E),\ \forall
 x\in D,\ 
\forall n\in \mathbb{N},
\label{2015-09-14:08}
  \\
  D\rightarrow L^1(G,E),\ x\mapsto \tilde f_n(x,\cdot), \ \mbox{belongs to }L^{p,q}\left(D,L^1(G,E)\right)\label{2015-09-14:09}\\
  \lim_{n\rightarrow \infty}\int_{D_2} \left( \int_{D_1}\left|\tilde f_n((x_1,x_2),\cdot)-\tilde f((x_1,x_2),\cdot)\right|_{L^1(G,E)}
^p
\nu_1(dx_1,x_2)
 \right) 
^{q/p}
\nu_2(dx_2)=0.\label{2015-08-15:10}
\end{gather}
\end{lemma}

\begin{proof}
Since $f$ is Bochner integrable,  
without loss of generality 
we can assume 
that $f(G)$ is separable.
Then there exists 
a  sequence $\{f_n\}_{n\in \mathbb{N}}$ of $L^{p,q}_\mathcal{D}(E)$-valued simple
  functions such that 
  \begin{align}
    &\lim_{n\rightarrow \infty}f_n(y)=f(y)\ 
\mbox{in\ }L^{p,q}_{\mathcal{D}}(E),\ 
 \forall y\in G
    ,\label{2015-08-14:01}\\
    &\lim_{n\rightarrow \infty}
|f_n-f|_{L^1 \left( G,L^{p,q}_\mathcal{D}(E) \right) }=0\label{2015-08-14:11}.
  \end{align}
Each $f_n$ can be written in the form
\begin{equation}
  \label{eq:2015-09-14:00}
  f_n(y)=\sum_{i=1}^{M(n)}\mathbf{1}_{A^{n}_i}(y)\varphi_{n,i}\qquad \forall y\in G,
\end{equation}
where $M(n)\in \mathbb{N}$, $A^n_i\in \mathcal{G}$, and $\varphi_{n,i}$
is a fixed representant of its equivalence class
in $L^{p,q}_\mathcal{D}(E)$.
  For
$n\in \mathbb{N}$, define
\[
\tilde f_n\colon  \left( D\times G,
\mathcal{D} \otimes \mathcal{G} \right) \rightarrow E,\ (x
,y)\mapsto f_n(y)(x).
\]
By using  \eqref{eq:2015-09-14:00}, we have
the measurability of \eqref{2015-09-14:06},  and 
\eqref{2015-09-14:03},
\eqref{2015-09-14:04},
\eqref{2015-09-14:08},
\eqref{2015-09-14:09}, are immediately verified.

We claim that
the sequence $\{\tilde f_n\}_{n\in
  \mathbb{N}}$ is Cauchy in $L^1\left( D\times G,
E\right)$.
Indeed, since  $\varphi_{n,i}\in L^{p,q}_\mathcal{D}(E)$,
 we have
 $\tilde f_n\in L^1(D\times G,E)$, for every $n\in
\mathbb{N}$. Moreover, by H\"older's inequality, 
\[
\begin{split}
  \int_{D\times G}
&|\tilde f_n-\tilde f_m|_E  d(\nu \otimes \mu)=\\
&=
  \int_G \left(\int_{D_2} \left( \int_{D_1} |\tilde f_n((x_1,x_2),y)-\tilde f_m((x_1,x_2),y)|_E  \nu_1(dx_1,x_2)\right) \nu_2(dx_2)\right)\mu(dy)\\
  &\leq 
C(p,q)
\int_G 
 \left( \int_{D_2}
\left( \int_{D_1}
    |\tilde f_n((x_1,x_2),y)-\tilde f_m((x_1,x_2),y)|_E^p \nu_1(dx_1,x_2)\right)^{q/p} \nu_2(dx_2)\right) ^{1/q}\mu(dy)\\
&  = 
C(p,q) |f_n-f_m|_{ L^1\left(G,L^{p,q}_\mathcal{D}
(E)\right) },
\end{split}
\]
and the last member tends to $0$ as $n$ and $m$ tend to $\infty$, by
\eqref{2015-08-14:11}.  Then there exists
$\tilde f\in L^1(D\times G,E)$ such that, after replacing $\{\tilde
f_n\}_{n\in \mathbb{N}}$ by a subsequence if necessary,
\begin{align}
  & \lim_{n\rightarrow \infty}\tilde f_n(x,y)=\tilde
  f(x,y)
   &&\forall (x,y)\in (D\times G)\setminus  N \label{2015-08-14:02}\\
  & \lim_{n\rightarrow \infty}\tilde f_n=\tilde f
    &&\mbox{in }L^1(D\times G,E),\ \ \ \label{2015-08-14:03}
\end{align}
where $ N$ is a $\nu \otimes  \mu$-null set.
We redefine $\tilde f$ on $N$ by $\tilde f(x,y)\coloneqq 0$ for
$(x,y)\in  N$.
After such a redefinition, the partial results of the theorem  till now proved
 still hold true.

By \eqref{2015-08-14:02}, 
since we can assume that each $\varphi_{n,i}$ has separable range,
we see that the range of $\tilde f$ is separable.
By measurability of sections of real-valued measurable functions and by Pettis's measurability theorem (use the fact that the range of $\tilde f$ is separable and then use Hahn-Banach theorem to extend continuous linear functionals on the space generated by the range of $\tilde f$ to the whole space $E$),
 we have that
\[
 \left( D,\mathcal{D}\right) \rightarrow E,\ x\mapsto \tilde f(x,y')\ \ \ \mbox{and}\ \ \ 
 \left( G,\mathcal{G} \right) \rightarrow E, \ y\mapsto \tilde f(x',y)
\]
are measurable, for all $y'\in G$ and all $x'\in D$. 
Since
\begin{multline}
  \label{eq:2015-08-15:14}
\int_G \liminf_{m\rightarrow \infty}\left|\tilde f(\cdot,y)-\tilde f_m(\cdot,y)\right|_{p,q}\mu(dy)\\
\leq \lim_{m\rightarrow \infty}\int_G
\left(
\int_{D_2}
 \left( 
\int_{D_1}
\left|\tilde f((x_1,x_2),y)-\tilde f_m((x_1,x_2),y)\right|^p_E
\nu_1(dx_1,x_2)
\right)^{q/p}
\nu_2(dx_2)
 \right)^{1/q}
\mu(dy)\\
\leq
  \lim_{m\rightarrow\infty}\liminf_{n\rightarrow \infty}
|f_n-f_m|_{ L^1\left(G,L^{p,q}_\mathcal{D}(E)\right) }
=0,
\end{multline}
we have
\begin{equation}
  \label{eq:2015-09-14:07}
\liminf_{m\rightarrow \infty}\left|\tilde f(\cdot,y)-\tilde f_m(\cdot,y)\right|_{p,q}=0\qquad\mu\mbox{-a.e.\ }y\in G.
\end{equation}
By recalling that $\tilde f_n(\cdot,y)\in L^{p,q}_\mathcal{D}(E)$ for all $y\in G$,
\eqref{eq:2015-09-14:07} shows that the map 
$$
D\rightarrow E,\ x\mapsto \tilde f(x,y)
$$ belongs to
$L^{p,q}_\mathcal{D}(E)$
for all $y\in G\setminus N'$, where $N'$ is a $\mu$-null set.
We redefine $\tilde f$ on $ N'$ by $\tilde f(x,y)\coloneqq 0$ for
$(x,y)\in  D\times N'$.
Again, we notice
that
the partial results of the theorem till now proved still hold  true after the redefinition on $D\times N_1$.
In addition,
$$
\forall y\in G,\ \mbox{the map\ }
D\rightarrow E,\ x\mapsto
 \tilde f(x,y),\ \mbox{belongs to\ } L^{p,q}_\mathcal{D}(E).
$$
This provides \eqref{2015-08-14:21}.
Moreover, 
since $N'$ can be chosen such that 
\eqref{eq:2015-09-14:07} holds for all $y\in G\setminus N'$
and since $G\rightarrow L^{p,q}_{\mathcal{D}}(E),\ y \mapsto \tilde f_n(\cdot,y)=f_n(y)$, is measurable, for all $n\in \mathbb{N}$,
also 
\eqref{2015-08-14:22} is proved.
From
the last inequality of
 \eqref{eq:2015-08-15:14}, 
  \eqref{2015-08-15:06} follows.
From
\eqref{2015-08-14:01}
and 
\eqref{eq:2015-09-14:07},
 \eqref{2015-08-14:15} follows as well. 

By H\"older's inequality, 
we have $|\tilde f|_{L^1(D\times G,E)}\leq  C(p,q)|\tilde f|_{L^1(G,L^{p,q}_\mathcal{D}(E))}<\infty$.
Then, 
after
 redefining
 $\tilde f$ on a  set $ N''\times G$, where $N''$ is a 
$\nu$-null set, 
 by $\tilde f(x,y)\coloneqq 0$ for
$(x,y)\in N''\times G$, 
we have
$$
\ \forall x\in D,\ \mbox{the map\ }
G\rightarrow E, \ y\mapsto \tilde f(x,y), \mbox{ belongs to }   L^1(G,E).
$$
This provides \eqref{2015-08-14:16}.
By applying Minkowski's inequality for integrals twice
(see \cite[p.\ 194, 6.19]{Folland1999}), we have
\begin{equation*}
  \begin{split}
    \lim_{n\rightarrow \infty}&
 \left( 
    \int_{D_2}
    \left(
      \int_{D_1}
      \left(
        \int_G
        |\tilde f((x_1,x_2),y)-\tilde f_n((x_1,x_2),y)|_E\mu
        (dy)\right)^p
      \nu_1(dx_1,x_2)\right)^{q/p}
    \nu_2(dx_2)
 \right) ^{1/q}
    \\
&\leq
    \lim_{n\rightarrow \infty}
 \left( 
    \int_{D_2}
    \left(
      \int_G
      \left(
        \int_{D_1}
        |\tilde f((x_1,x_2),y)-\tilde f_n((x_1,x_2),y)|^p_E\nu_1        (dx_1,x_2)\right)^{1/p}
      \mu(dy)\right)^{q}
    \nu_2(dx_2)
 \right) ^{1/q}
    \\
&\leq
    \lim_{n\rightarrow \infty}
    \int_G
    \left(
      \int_{D_2}
      \left(
        \int_{D_1}
        |\tilde f((x_1,x_2),y)-\tilde f_n((x_1,x_2),y)|^p_E\nu_1        (dx_1,x_2)\right)^{q/p}
      \nu_2(dx_2)\right)^{1/q}
    \mu(dy).
\end{split}
\end{equation*}
Since the latter member tends to $0$ because of
the second inequality in
 \eqref{eq:2015-08-15:14}, 
the estimate above
provides
\eqref{2015-09-14:10} and  \eqref{2015-08-15:10},
after  redefining $\tilde f$ 
 on a set $N'''\times G$,
where $N'''$ is a
 suitably chosen  $\nu$-null set,
 by $\tilde f(x,y)\coloneqq 0$ for
$(x,y)\in   N'''\times G$.
\end{proof}

Let 
$T>0$
 and let $\mathcal{B}_T$  be a short notation for the Borel $\sigma$-algebra   $\mathcal{B}_{[0,T]}$
on $[0,T]$.
We recall that, if $\mathcal{T}$ is a topological space, then $\mathcal{B}_{\mathcal{T}}$ denotes the Borel $\sigma$-algebra of $\mathcal{T}$ (\footnote{No topological space will be denoted by $T$, hence there will not be any confusion with $\mathcal{B}_T$.}).
Let 
 $\left(\Omega,\mathcal{F},\mathbb{F}\coloneqq \{\mathcal{F}_t\}_{t\in[0,T]},\mathbb{P}\right)$ 
be a complete filtered probability space.
We endow the product space $\Omega_T\coloneqq \Omega \times [0,T]$  with the $\sigma$-algebra ${\mathcal{P}_T}$ of predictables sets associated to the filtration $\mathbb{F}$
 and the measurable space $ \left( \Omega_T,{\mathcal{P}_T} \right)$ 
with the product measure $ \mathbb{P} \otimes  m$, where $m$ denotes the Lebesgue's measure.
We need to introduce some further notation.
  \begin{itemize}
\itemsep=-1mm
\item 
\vskip-5pt
$F$ is a Banach space;
  \item $\mathbb{T}\subset B_b([0,T],F)$ is a closed subspace (with respect to the norm $|\cdot|_\infty$) such that
    \begin{equation}
      \label{eq:2017-05-05:01}
      \mathbb{T}\times  [0,T]\rightarrow F,\ (\mathbf{x},t) \mapsto \mathbf{x}(t);
    \end{equation}
is Borel measurable, when $\mathbb{T}\times [0,T]$  is endowed with the product $\sigma$-algebra $\mathcal{B}_\mathbb{T} \otimes \mathcal{B}_{[0,T]}$ (and not just with the Borel $\sigma$-algebra of the product topology!).
\item 
  $\mathcal{P}'$ is a given sub-$\sigma$-algebra of $\mathcal{F}_T \otimes \mathcal{B}_T$ such that, for all $A\in \mathcal{F}_T$ with $\mathbb{P}(A)=0$, $A\times [0,T]\in \mathcal{P}'$.
\item 
 $\mathcal{L}^0_{\mathcal{P}'}(\mathbb{T})$
is  the vector space of measurable functions
$$
X\colon (\Omega_T,\mathcal{P}')\rightarrow F
$$ 
such that, for $\mathbb{P}$-a.e.\ $\omega\in \Omega$, the path
$$
X(\omega)\colon [0,T]\rightarrow  F,\ t \mapsto X_t(\omega)
$$
belongs to $\mathbb{T}$, and the $\mathbb{P}$-a.e.\ defined 
map
\begin{equation}
  \label{eq:2017-05-05:00}
  (\Omega,\mathcal{F}_T)\rightarrow \mathbb{T},\ \omega \mapsto X(\omega)
\end{equation}
is measurable, when $\mathbb{T}$ is endowed with the Borel $\sigma$-algebra induced by the norm $|\cdot|_\infty$.
\item \label{not:2017-05-07:08}
For $r\in[1,\infty)$,  $\mathcal{L}^r_{\mathcal{P}'}(\mathbb{T})$ denotes the space of (equivalence classes of)
 $X\in \mathcal{L}^0_{\mathcal{P}'}(\mathbb{T})$ such that
\eqref{eq:2017-05-05:00} has separable range and
$$
|X|
_{
  \mathcal{L}^r
  _{\mathcal{P}'}
  (\mathbb{T})
}
\coloneqq  \left( \mathbb{E} \left[ |X|^r_\infty \right] \right) ^{1/r} <\infty.
$$
Then
$(\mathcal{L}^r_{\mathcal{P}'}(\mathbb{T}),
|\cdot|_{\mathcal{L}^r_{\mathcal{P}'}(\mathbb{T})})$ is a Banach space.
\end{itemize}

\begin{remark}
  The space $\mathbb{T}$ can be e.g.\ $C_b([0,T],F)$, because in such a case \eqref{eq:2017-05-05:01} is continuous, hence measurable.
This permits also to consider $\mathbb{T}$ as the space of left-limited right-continuous functions, because,
if $\varphi$ is real valued and continuous with support $[0,1]$ and if
 $\varphi_\epsilon(t)=\epsilon^{-1}\varphi(\epsilon^{-1}t)$,
then 
$\varphi_\epsilon* \mathbf{x}$ converges pointwise to $\mathbf{x}$ everywhere on $[0,T]$ as
\mbox{$\epsilon\rightarrow 0^+$}, after extending
 $\mathbf{x}$  by continuity beyond $T$.
We finally  observe that \eqref{eq:2017-05-05:01} is measurable whenever $\mathbb{T}$ is separable:
this comes from a straightforward application 
of \cite[Lemma 4.51]{Aliprantis2006}.
\end{remark}

We now provide
 the main result of this section.

\begin{theorem}[Stochastic Fubini's theorem] 
\label{2015-09-21:01}
Let $p,q,r\in[1,\infty)$, $g\in L^1(G,L^{p,q}_{\mathcal{D}}(E))$.
Let 
\begin{equation*}
   \mathfrak L\colon L^{p,q}_{\mathcal{D}} \left(E \right)
  \rightarrow 
\mathcal{L}^r_{\mathcal{P}'}(\mathbb{T})
\qquad
\end{equation*}
be a linear and continuous operator.
Then there exist measurable functions
\begin{gather*}
  X_1\colon  \left( D\times G,  \mathcal{D} \otimes \mathcal{G} \right) \rightarrow E\\
  X_2\colon \left(  \Omega_T\times G , 
 \left(  \mathcal{F}_T \otimes \mathcal{B}_T \right)  \otimes \mathcal{G}  
 \right) \rightarrow F
\end{gather*}
such that
\begin{subequations}
  \begin{equation*}
     X_1(x,\cdot)\in L^1(G,E),\ \forall x\in D, \mbox{ and }
  X_2((\omega,t),\cdot)\in L^1(G,F),\ 
  \forall (\omega,t)\in \Omega_T
\end{equation*}
\vskip-24pt
\begin{equation*}
    D\rightarrow L^1(G,E), \ x\mapsto
  X_1(x,\cdot),\ \in
  L^{p,q}_{\mathcal{D}}(L^1(G,E))
\end{equation*}
\begin{equation*}
  \left(   \Omega_T,\mathcal{P}_T  \right)
\rightarrow L^1(G,F), \ (\omega,t)\mapsto
  X_2((\omega,t),\cdot),
  \mbox{ is measurable}
\end{equation*}
\begin{equation*}
  X_1(\cdot,y)\in L^{p,q}_{\mathcal{D}}(E),\ \forall y\in G
\end{equation*}
\begin{equation*}
  G\rightarrow L^{p,q}_{\mathcal{D}}(E),\ y \mapsto X_1(\cdot,y),\ \in L^1(G,L^{p,q}_{\mathcal{D}}(E))
\end{equation*}
\begin{equation*}
    X_1(\cdot,y)=g(y)
  \mbox{ in }L_{\mathcal{D}}^{p,q}(E)\mbox{ for } \mu\mbox{-a.e.\ }y\in G
\end{equation*}
\begin{equation*}
  X_2(\cdot,y)\in 
\mathcal{L}^r_{\mathcal{P}'}(\mathbb{T}),
\ \forall y\in G
\end{equation*}
\vskip-18pt
\begin{equation}\label{2017-05-07:00}
  X_2(\cdot,y)=\mathfrak Lg(y)\mbox{ in }
\mathcal{L}^r_{\mathcal{P}'}(\mathbb{T}),\
  \mu\mbox{-a.e.\ }y\in G
\end{equation}
\end{subequations}
and such that 
\begin{equation}
  \begin{split}
      \label{eq:2015-09-15:00}
\mbox{for $\mathbb{P}$-a.e.\ $\omega\in\Omega$,\ }
(\mathfrak LY)(\omega,t)&=\int_G X_2((\omega,t),y)\mu(dy),\ \forall t\in[0,T],
\end{split}
\end{equation}
where
\[
Y(x)\coloneqq \int_GX_1(x,y)\mu(dy),\ \forall
x\in D.
\]
\end{theorem}
\begin{proof}
 By 
 Lemma~\ref{2015-08-15:00} ,
there exist measurable functions
  \begin{gather*}
    \tilde f\colon  \left(  D\times G, \mathcal{D} \otimes \mathcal{G}\right)\rightarrow E,\\
    \tilde  f _n\colon \left( D\times G, \mathcal{D} \otimes \mathcal{G}\right)
\rightarrow E,\qquad n\in \mathbb{N}
  \end{gather*}
satisfying  \eqref{2015-08-14:16}--\eqref{2015-08-15:10}.
For $n\in \mathbb{N}$,  $\tilde  f _n$ has the form
$$
\tilde  f _n(x,y)=\sum_{i=1}^{M(n)}\mathbf{1}_{A^n_i}(y)\varphi_{n,i}(x)\qquad\forall x\in  D,\ 
\forall y\in G,
$$
where $\varphi_{n,i}$ is a fixed representant of its class in $L^{p,q}_{\mathcal{D}}(E)$.
For all $n\in \mathbb{N}$, the function $\tilde f_n^{(\mu)}$ defined by 
$$
\tilde  f _n^{(\mu)}\colon D\rightarrow E\ x\mapsto \int_G\tilde  f _n(x,y)\mu(dy)=\sum_{i=1}^{M(n)}\varphi_{n,i}(x
)\mu(A^n_i)
$$
belongs to $L^{p,q}_{\mathcal{D}}(E)$. 
Then, if we define
$$
\tilde  f ^{(\mu)}\colon D\rightarrow E,\ x\mapsto\int_G\tilde  f (x,y)\mu(dy),
$$
due to \eqref{2015-08-15:10}, we have
\begin{equation}
  \label{eq:2015-08-15:11}
  \lim_{n\rightarrow\infty}\tilde  f ^{(\mu)}_n=\tilde  f ^{(\mu)}
\  \mbox{in }L^{p,q}_{\mathcal{D}}(E).
\end{equation}
By linearity of $ \mathfrak L$, we have
\begin{equation}
  \label{eq:2015-09-16:00}
\mathfrak L\tilde  f ^{(\mu)}_n=
\sum_{i=1}^{M(n)}\mu(A^n_i)\mathfrak L\varphi_{n,i}  \ \textrm{in}\ \mathcal{L}^r_{\mathcal{P}'}(\mathbb{T}).
\end{equation}
By continuity of  $\mathfrak L$,
\eqref{eq:2015-08-15:11} and \eqref{eq:2015-09-16:00} give
\begin{equation}
  \label{eq:2015-08-15:04}
  \lim_{n\rightarrow\infty}
\sum_{i=1}^{M(n)}\mu(A^n_i)\mathfrak L\varphi_{n,i}=
\mathfrak L \tilde  f ^{(\mu)}
\ \mbox{in }\mathcal{L}^r_{\mathcal{P}'}(\mathbb{T}).
\end{equation}
For $n\in \mathbb{N}$, 
we now consider  the measurable function
\[
\tilde  f ^{(\mathfrak L )}_n\colon  \left(  \Omega_T\times G,{\mathcal{P}'} \otimes \mathcal{G} \right)  \rightarrow F, \ ((\omega,t),y)\mapsto \sum_{i=1}^{M(n)}\mathbf{1}_{A^n_i}(y) \left( \mathfrak L \varphi_{n,i} \right) (\omega,t)
\]
where here $\mathfrak L \varphi_{n,i}$ is a fixed representant of its class in $\mathcal{L}^r_{\mathcal{P}'}(\mathbb{T})$.
For 
all $y\in G$, $\tilde  f ^{(\mathfrak L)}_n(\cdot,y)$ is a representant of the class of $\mathfrak L  \left( \tilde  f _n(\cdot,y) \right) $ in
$\mathcal{L}^r_{\mathcal{P}'}(\mathbb{T})$.
Moreover,
\[
\int_G \tilde  f ^{(\mathfrak L)}_n((\omega,t),y)\mu(dy)=\sum_{i=1}^{M(n)}\mu(A^n_i)\left(\mathfrak L \varphi_{n,i}\right)(\omega,t)\qquad \forall (\omega,t)\in \Omega_T,\ \forall n\in \mathbb{N}.
\]
By \eqref{eq:2015-09-16:00}, we obtain
\begin{equation}
  \label{eq:2015-08-15:03}
\mbox{for $\mathbb{P}$-a.e.\ $\omega\in \Omega$,\ }  \int_G \tilde  f ^{(\mathfrak L )}_n((\omega,t),y)\mu(dy)=
 (\mathfrak L \tilde f _n^{(\mu)} )(\omega,t)
\ \forall t\in[0,T].
\end{equation}
We now show that we can pass to the limit in \eqref{eq:2015-08-15:03}.
By \eqref{2015-08-15:06}, 
\[
\lim_{n\rightarrow \infty}\int_G\left|\tilde  f _n(\cdot,y)-\tilde  f (\cdot,y)\right|_{L^{p,q}_{\mathcal{D}}(E)}\mu(dy)=0,
\]
hence, by continuity of $\mathfrak L $,
\begin{equation}
  \label{eq:2015-08-15:12}
  \lim_{n\rightarrow \infty}\int_G\left| \mathfrak L   \left( \tilde  f _n(\cdot,y) \right) -\mathfrak L  \left( \tilde  f (\cdot,y) \right) \right|_{\mathcal{L}^r_{\mathcal{P}'}(\mathbb{T})}\mu(dy)=0.
\end{equation}
Since $\mathcal{L}^r_{\mathcal{P}'}(\mathbb{T})$ is
a closed subspace of
$$
L^r(\Omega,\mathbb{T})
\coloneqq 
L^r \left( (\Omega,\mathcal{F}_T,\mathbb{P}), 
\left( \mathbb{T}
,|\cdot|_\infty \right)   \right),
$$ 
the map
\begin{equation}
  \label{eq:2015-09-15:07}
  (G,\mathcal{G})\rightarrow L^r(\Omega,\mathbb{T})
,\ y\mapsto \mathfrak L  \left( \tilde  f (\cdot,y) \right) 
\end{equation}
is measurable and integrable (the range of \eqref{eq:2015-09-15:07} is separable). 
By applying Lemma \ref{2015-08-15:00} again, now to 
\eqref{eq:2015-09-15:07},
we have that there exists a measurable function
\begin{equation}
  \label{eq:2017-05-05:02}
  g\colon  \left( \Omega\times G, \mathcal{F}_T \otimes \mathcal{G} \right) \rightarrow \mathbb{T}
\end{equation}
such that,
for some $A\in \mathcal{G}$ with $\mu(A^c)=0$,
\begin{gather}\label{2015-09-15:08}
  g(\cdot,y)=\mathfrak L  \left( \tilde  f (\cdot,y) \right) \ \mbox{in }L^r(\Omega,\mathbb{T}
 ),\ \forall y\in A.
\end{gather}
Define
\begin{equation*}
  X_2((\omega,t),y)\coloneqq 
  \begin{dcases}
    g(\omega,y)(t)
 & \forall ((\omega,t),y)\in  \Omega_T\times A\\
    0 &\textrm{otherwise.}
  \end{dcases}
\end{equation*}
Notice that,  since $\mathfrak  L \left( \tilde f(\cdot,y) \right) $ is $\mathcal{P}'$-measurable for all $y\in G$ (by definition of $ \mathfrak L$) and since $\mathcal{P}'$ contains the sets $N\times [0,T]$ when $N\in \mathcal{F}_T$ and $\mathbb{P}(N)=0$, we have,
by 
\eqref{2015-09-15:08},
 that $X_2(\cdot,y)$ is $\mathcal{P}'$-measurable for all $y\in G$.
Moreover, since the evaluation map \eqref{eq:2017-05-05:01} is assumed to be measurable, by measurability of \eqref{eq:2017-05-05:02}  and by definition of $X_2$
 we have that
$$
X_2\colon  \left(   \Omega_T\times G,   \left( 
 \mathcal{F}_T
\otimes
\mathcal{B}_T 
 \right) \otimes \mathcal{G}   \right)  \rightarrow F
$$
is measurable. 
By \eqref{eq:2015-08-15:12}, we can write
\begin{equation}
  \label{eq:2017-05-11:00}
  \begin{split}
  \lim_{n\rightarrow \infty}\int_\Omega &\left(\int_G
\sup_{t\in [0,T]}\left|\tilde  f _n^{(\mathfrak L )}((\omega,t),y)-
X_2((\omega,t),y)
  \right|_{F} \mu(dy)\right)\mathbb{P}(d\omega)\\
&\leq
  \lim_{n\rightarrow \infty}\int_G
\left(\int_\Omega
\sup_{t\in[0,T]}
\left| \tilde f _n^{(\mathfrak L )}((\omega,t),y)-
X_2((\omega,t),y)\right|_{F}^r\mathbb{P}(d\omega)\right)^{1/r}\mu(dy)\\
&=  \lim_{n\rightarrow \infty}\int_G\left|\mathfrak L  \left( \tilde  f _n(\cdot,y) \right) -\mathfrak L  \left( \tilde  f (\cdot,y) \right) \right|_{
\mathcal{L}_{\mathcal{P}'}^r(\mathbb{T})}\mu(dy)=0,
\end{split}
\end{equation}
where the measurability of 
$|\tilde f_n^{(\mathfrak L )}((\omega,\cdot),y)-X_2(
(\omega,\cdot),y)|_\infty$, jointly in  $(\omega,y)$, is due to the measurability of \eqref{eq:2017-05-05:02}, to the definition of $X_2$, and to the definition of $\tilde f_n^{( \mathfrak L)}$.
By \eqref{eq:2017-05-11:00}, by considering a subsequence if necessary,  it follows that
\begin{equation}
  \label{eq:2015-08-15:13}
  \lim_{n\rightarrow \infty}
\sup_{t\in[0,T]}\left|\int_G \tilde  f ^{(\mathfrak L )}_n((\omega,t),y)\mu(dy)
-
\int_G X_2((\omega,t),y) \mu(dy)\right|_F=0\qquad\mathbb{P}\mbox{-a.e.\ }\omega\in \Omega.
\end{equation}
By
\eqref{eq:2015-09-16:00},
\eqref{eq:2015-08-15:04},
\eqref{eq:2015-08-15:03}, and
\eqref{eq:2015-08-15:13}, we conclude that, 
for $\mathbb{P}$-a.e.\ $\omega\in \Omega$,
$$
    ( \mathfrak L \tilde  f ^{(\mu)}  ) (\omega,t) =\int_G X_2((\omega,t),y) \mu(dy),\ \forall t\in [0,T],
$$
which provides \eqref{eq:2015-09-15:00},
after defining $X_1\coloneqq \tilde  f $.
\end{proof}

\section{Stochastic convolution}\label{2017-05-10:00}



One of the contents of Theorem~\ref{2015-09-21:01} is 
the existence of the jointly measurable function $X_2$, whose sections $X_2(\cdot,y)$ coincide with the ``stochastic integral'' $ \mathfrak L g(y)$, for a.e.\ $y$.
This fact permits to obtain a jointly measurable version of a stochastic convolution, as we will explain
in the present section.

\bigskip
Let us recall/introduce the following notation.
  We consider separable Hilbert spaces $H$ and $U$, with scalar product
 $\langle \cdot,\cdot\rangle_H$ and $\langle \cdot,\cdot\rangle_U$,
respectively.
  \begin{itemize}
\itemsep=-1mm
\item 
 $L_2(U,H)$ denotes the space of Hilbert-Schmidt linear operators from $U$ into $H$.
\end{itemize}
Let $E$ be a Banach space.
\begin{itemize}
\itemsep=-1mm
\item If $E$ is a Banach space,  $\leb^0_{\mathcal{P}_T}(E)$ denotes the space of $E$-valued $ {\mathcal{P}_T}/\mathcal{B}_E $-measurable processes $\Phi\colon \Omega_T\rightarrow E$, for which there exists $N\in \mathcal{P}_T$ with $\mathbb{P} \otimes m(N)=0$ such that $X(\Omega_T\setminus N)$ is separable.
Two
 processes are equal in $L^0_{\mathcal{P}_T}(E)$ if they coincides $ \mathbb{P} \otimes m$-a.e..
The space $L^0_{\mathcal{P}_T}(E)$ is a complete metrizable space when endowed with the  topology induced by the convergence in measure (see \cite[Sec.\ 5.2]{Malliavin1995}).
  \item 
 $\leb^0_{{\mathcal{P}_T} \otimes \mathcal{B}_T}(E)$ denotes the space of (equivalence classes of) $E$-valued $ \mathcal{P}_T \otimes {\mathcal{B}_T}/\mathcal{B}_E $-measurable processes $\zeta\colon
 \Omega_T\times [0,T]\rightarrow E$, with separable range,
up to a modification on a $(\mathbb{P} \otimes m )  \otimes m$-null set if necessary.
Two processes are  equal in $L^0_{ \mathcal{P}_T \otimes \mathcal{B}_T}(E)$ if they coincides $ (\mathbb{P} \otimes  m) \otimes m$-a.e..
 $L^0_{ \mathcal{P}_T  \otimes \mathcal{B}_T}(E)$ is 
endowed  with the metrizable complete vector topology induced by the convergence in measure.
\item
For $p,q\in[1,\infty)$,  $L_{\mathcal{P}_T}^{p,q}(E)$ denotes the
 subspace
 of  $L^0_{\mathcal{P}_T}(E)$
 whose members $X$ satisfy
$$
|X|_{p,q}=
 \left(\int_0^T \left( \mathbb{E} \left[ |X_t|_E^p \right] \right)^{q/p} dt\right) ^{1/q}<\infty.
$$
$(L^{p,q}_{\mathcal{P}_T}(E),|\cdot|_{p,q})$ is a Banach space.
The space $L^{p,q}_{\mathcal{F}_T \otimes \mathcal{B}_T}(E)$ is defined similarly to 
$L^{p,q}_{\mathcal{P}_T}(E)$, after replacing $\mathcal{P}_T$ by $\mathcal{F}_T \otimes \mathcal{B}_T$.
We use the  notation $L^p_{\mathcal{P}_T}(E)$,
 $L^p_{\mathcal{F}_T \otimes \mathcal{B}_T}(E)$, for $L^{p,p}_{\mathcal{P}_T}(E)$, $L^{p,p}_{\mathcal{F}_T \otimes B_T}(E)$, respectively.
\item 
For $p,q,r\in [1,\infty)$.
$L^{p,q,r}_{\mathcal{P}_T \otimes {\mathcal{B}_T}}(E)$ denotes the space containing those $\zeta\in L^0_{\mathcal{P}_T \otimes {\mathcal{B}_T}}(E)$ such that
  \begin{equation}
    \label{eq:2015-11-06:02}
    |\zeta|_{p,q,r}\coloneqq 
     \left(\int_0^T \left(  \int_0^T  
      \left( \mathbb{E} \left[    |\zeta((\omega,s),t)|_{
E}^p  \right] \right) ^{q/p}ds  \right)^{r/q}dt  \right) ^{1/r}<\infty.
\end{equation}
$(L^{p,q,r}_{\mathcal{P}_T \otimes {\mathcal{B}_T}}(E),|\cdot|_{p,q,r})$ is a Banach space.
\end{itemize}

\subsection{Jointly measurable version}

In this section we employ Theorem~\ref{2015-09-21:01} to obtain jointly measurable versions of stochastic integrals
(represented, as in the previous section, by a generic 
continuous linear
operator $ \mathfrak I$) depending on parameter.

We will often need to consider sections of measurable functions and their measurability with respect to some codomains.
We begin with the following lemma.

\begin{lemma}
\label{2015-12-29:01}
    Let $\zeta\in L^0_{
    {\mathcal{P}_T} \otimes \mathcal{B}_T}(L_2(U,H))$.  
Then
      \begin{equation}
        f_\zeta\colon [0,T]\rightarrow L^0_{  \mathcal{P}_T} \left( L_2(U,H) \right) ,\ t \mapsto \zeta(\cdot,t)
      \end{equation}
is  measurable.
\end{lemma}
\begin{proof}
Let us first suppose that $U=H=\mathbb{R}$, hence $L_2(U,H)=\mathbb{R}$.
Define
$$
\mathcal{C}\coloneqq  \left\{ A\in  \mathcal{P}_T \otimes \mathcal{B}_T\ \mbox{s.t.}\ f_{\mathbf{1}_A}\ \mbox{is measurable}\right\}.
$$
It is clear that the rectangles of the form $B\times C$, with  $B\in \mathcal{P}_T$
and $C\in \mathcal{B}_T$, belong to $\mathcal{C}$, because $f_{\mathbf{1}_{B\times C}}$ assumes only the two values $0$ ans $\mathbf{1}_B$ on $\Omega_T\setminus C$ and on $B$, respectively.
If $A\in \mathcal{C}$, $B\in \mathcal{C}$, $B\subset A$, then $f_{\mathbf{1}_{A\setminus B}}=f_{\mathbf{1}_A}-f_{\mathbf{1}_B}$ is measurable, and then $A\setminus B\in \mathcal{C}$.
If $\{A_n\}_{n\in \mathbb{N}}\subset \mathcal{C}$ is an increasing sequence, then $f_{\mathbf{1}_{\cup_{n\in \mathbb{N}}A_n}}(t)=\lim_{n\rightarrow \infty}f_{\mathbf{1}_{A_n}}(t)$ in $L^0_{\mathcal{P}_T}(\mathbb{R})$ for all $t\in[0,T]$, hence $\bigcup_{n\in \mathbb{N}}A_n\in \mathcal{C}$.
This shows that $\mathcal{C}$ is a $\lambda$-class containing the rectangles $B\times C$, with $B\in \mathcal{P}_T$ and $C\in \mathcal{B}_T$, hence $\mathcal{P}_T \otimes  \mathcal{B}_T\subset \mathcal{C}$.
By linearity and by 
monotone convergence,
we have that $f_\zeta$ is measurable for all $\zeta\in L^0_{\mathcal{P}_T \otimes \mathcal{B}_T}(\mathbb{R})$.

Now let $U$, $H$, be generic separable Hilbert spaces
 and let
$\{\varphi_n\}_{n\in \mathbb{N}}$ be an orthonormal basis for $L_2(U,H)$ (we consider the case $\dim L_2(U,H)=\infty$; the case $<\infty$ is similar).
If $\zeta\in L^0_{\mathcal{P}_T \otimes \mathcal{B}_T}(L_2(U,H))$, then, for all $t\in[0,T]$,
$$
f_\zeta(t)(\omega,s)= \sum_{n\in \mathbb{N}}\langle
\varphi_n,
 \zeta((\omega,s),t)\rangle_{L_2(U,H)}\varphi_n
=\sum_{n\in \mathbb{N}}
f_{\langle
\varphi_n,
 \zeta\rangle_{L_2(U,H)}\varphi_n}(t)(\omega,s)\qquad \forall (\omega,s)\in \Omega_T.
$$
From the first part of the proof, 
$f_{\langle
\varphi_n,
 \zeta\rangle_{L_2(U,H)}\varphi_n}$ is measurable, after the identification $L^0_{\mathcal{P}_T}(\mathbb{R})=L^0_{\mathcal{P}_T}(\mathbb{R}\varphi_n)$ and the continuous, hence measurable, embedding 
$$
L^0_{\mathcal{P}_T}(\mathbb{R}\varphi_n)\subset L^0_{\mathcal{P}_T}(L_2(U,H)).
$$
 We  conclude that $f_\zeta$ is measurable, because it is the pointwise limit of the sequence
$$
\left\{\sum_{n=0}^Nf_{\langle
\varphi_n,
 \zeta\rangle_{L_2(U,H)}}\varphi_n\colon [0,T]\rightarrow L^0_{\mathcal{P}_T}(L_2(U,H))\right\}_{N\in \mathbb{N}}.
\eqno\qed
$$
\let\qed\relax
\end{proof}

\smallskip
\begin{remark}\label{2015-12-29:02}
If $p,q\in [1,\infty)$,  the map
$$
|\cdot|_{p,q}\colon  L^0_{ \mathcal{P}_T}(L_2(U,H))\rightarrow [0,\infty],\ \xi \mapsto |\xi|_{p,q}
$$
is lower-semicontinuous (Fatou's Lemma).
By Lemma \ref{2015-12-29:01}, if $\zeta\in L^0_{\mathcal{P}_T   \otimes \mathcal{B}_T}(L_2(U,H))$, then
      \begin{equation}
        f_\zeta\colon [0,T]\rightarrow L^0_{  \mathcal{P}_T} \left( L_2(U,H) \right) ,\ t \mapsto \zeta(\cdot,t)
      \end{equation}
is  measurable.
By combining $f_\zeta$ with $|\cdot|_{p,q}$, we have that the set 
\begin{equation}
  \label{eq:2015-12-29:03}
    B_\zeta\coloneqq 
\left\{ t\in[0,T]\colon 
      \zeta(\cdot,t)\in L^{p,q}_{{\mathcal{P}_T}}(L_2(U,H)) \right\}
  \end{equation}
  is a Borel set.
\end{remark}

\bigskip
Clearly the set $B_\zeta$ defined
in Remark \ref{2015-12-29:02}
 depends on the representant of $\zeta$ chosen in $L^0_{\mathcal{P}_T \otimes \mathcal{B}_T}(L_2(U,H)$. Hereafter,
 whenever a notion associated to some function $f$ belonging to some quotient space of mesurable functions
 is pointwise dependent, we mean that the notion is actually associated to a chosen representant $f$.

 \begin{notation}
   In what follows, we will always use the notation $B_\zeta$ for the set defined by \eqref{eq:2015-12-29:03}, when $\zeta\in L^0_{\mathcal{P}_T \otimes \mathcal{B}_T}(L_2(U,H))$.
In the notation, we omit the dependence of $B_\zeta$ on $p,q$, as it will be always clear from the context.
\end{notation}

\smallskip
The next result is a variant of
 Lemma~\ref{2015-12-29:01}
for 
$L^{p,q,r}_{\mathcal{P}_T \otimes  \mathcal{B}_T}(L_2(U,H))$. It will be used to derive jointly measurable versions of stochastic convolutions.


\begin{lemma}
  \label{2016-01-10:00}
  Let $p,q,r\in[1,\infty)$ and let $\zeta\in L^{p,q,r}_{\mathcal{P}_T \otimes  \mathcal{B}_T}(L_2(U,H))$.
  Let $B_\zeta$ be the Borel set defined by  \eqref{eq:2015-12-29:03}.
Then $m([0,T]\setminus B_\zeta)=0$ and
$$
f_\zeta\colon B_\zeta\rightarrow L^{p,q}_{\mathcal{P}_T}(L_2(U,H)),\ t \mapsto \zeta (\cdot,t)
$$
is Borel measurable.
\end{lemma}
\begin{proof}
It is clear that $m([0,T]\setminus B_\zeta)=0$,
because $\zeta\in L^{p,q,r}_{\mathcal{P}_T \otimes \mathcal{B}_T}(L_2(U,H))$ and then $\zeta(\cdot,t)\in L^{p,q}_{\mathcal{P}_T}(L_2(U,H))$ for $m$-a.e.\ $t\in [0,T]$.
By redefining $\zeta((\omega,s),t)\coloneqq 0$ for $((\omega,s),t)\in  \Omega_T\times [0,T]$, $t\in [0,T]\setminus B_\zeta$, we can assume that $B_\zeta=[0,T]$.
In such a case, to show that $f_\zeta$ is Borel measurable, we argue as in the proof of
 Lemma \ref{2015-12-29:01}, 
after replacing $L^0_{\mathcal{P}_T \otimes \mathcal{B}_T}$ by $L^{p,q,r}_{\mathcal{P}_T \otimes  \mathcal{B}_T}$ and $L^0_{\mathcal{P}_T}$ by $L^{p,q}_{ \mathcal{P}_T}$.
\end{proof}

For $p,q,r\in[1,\infty)$, let 
\begin{equation}
  \label{eq:2017-05-07:06}
 \mathfrak I\colon  L^{p,q}_{\mathcal{P}_T}(L_2(U,H))\rightarrow \mathcal{L}^r_{\mathcal{F}_T \otimes \mathcal{B}_T}(\mathbb{T})
\end{equation}
be a linear and continuous operator,
where $\mathcal{L}^r_{\mathcal{F}_T \otimes \mathcal{B}_T}(\mathbb{T})$ is defined as in Section~\ref{2016-01-19:05} (p.\ ~\pageref{not:2017-05-07:08}), with $\mathcal{P}'=\mathcal{F}_T \otimes \mathcal{B}_T$.

\bigskip
  Let $\zeta\in L^0_{\mathcal{P}_T \otimes \mathcal{B}_T}\blp L_2(U,H)\brp$ be a given representant of its class.
Our aim is to show that
there exists a  $(\omega,t)$-jointly measurable version of
 the family of random variables 
 \begin{equation}
   \label{eq:2017-05-07:10}
    \mathfrak I^\zeta_t
\coloneqq (\mathfrak I(\zeta(\cdot,t)))_{t\in B_\zeta},
\end{equation}
where $B_\zeta$ is defined by \eqref{eq:2015-12-29:03}.

 \begin{remark}\label{2016-01-09:02}
  Definition \ref{eq:2017-05-07:10} depends on the chosen representant $\zeta$.
If $\zeta=\zeta'$ in the space $L^0_{\mathcal{P}_T \otimes \mathcal{B}_T}(L_2(U,H))$, then $m(B_{\zeta}\bigtriangleup B_{\zeta'})=0$, and,
 due to the fact that $ \mathfrak I$ has values in $\mathcal{L}^r_{\mathcal{F}_T \otimes \mathcal{B}_T}(\mathbb{T})$, 
we have  $ \mathfrak I^\zeta_t= \mathfrak I^{\zeta'}_t$
 $\mathbb{P}$-a.e., for all $t\in B_\zeta\medcap B_{\zeta'}$.
\end{remark}


\begin{theorem}\label{2017-05-07:11}
Let $p,q,r\in [1,\infty)$, let 
$\zeta\in L^{p,q,r}_{\mathcal{P}_T \otimes \mathcal{B}_T}(L_2(U,H))$,
and let
$B_\zeta$ be the set defined by \eqref{eq:2015-12-29:03}.
  Then  
 there exists
a process 
\begin{equation}
  \label{eq:2017-05-07:03}
  \Sigma^\zeta\in L^r_{ \mathcal{F}_T \otimes \mathcal{B}_T}(F)
\end{equation}
 such that
 \begin{equation}
   \label{eq:2017-05-07:04}
   \mbox{for $m$-a.e.\ $t\in B_\zeta$,\ }
\Sigma^\zeta_t(\omega)=  (\mathfrak I(\zeta(\cdot,t)))_t(\omega)\ \mathbb{P}\textrm{-a.e.\ } \omega\in \Omega.
\end{equation}
Moreover, 
the map
\begin{equation}\label{2017-05-07:02}
\mathbf{J}\colon L^{p,q,r}_{\mathcal{P}_T \otimes \mathcal{B}_T}(L_2(U,H)) \rightarrow L^r_{\mathcal{F}_T \otimes \mathcal{B}_T}(F),\ \zeta \mapsto  \Sigma^\zeta
\end{equation}
is linear, continuous,
 uniquely determined by \eqref{eq:2017-05-07:03}, \eqref{eq:2017-05-07:04}.
The operator norm of $\mathbf{J}$ is bounded by the operator norm of $ \mathfrak I$.
\end{theorem}
\begin{proof}
We apply Theorem~\ref{2015-09-21:01}, with the following data:
\begin{itemize}
\itemsep=-1mm
\item $G=[0,T]$, $\mathcal{G}=\mathcal{B}_T$, $\mu=m$;
\item $D_1=\Omega$, $D_2=[0,T]$, $D=\Omega_T$, $\mathcal{D}=\mathcal{P}_T$, $\nu_1=\mathbb{P}$, $\nu_2=m$;
\item $E=L_2(U,H)$;
\item  $\mathfrak L= \mathfrak I$;
\item $g\colon [0,T]\rightarrow L^{p,q}_{\mathcal{P}_T}(L_2(U,H))$ defined by
  \begin{equation*}
    g(t)\coloneqq
    \begin{dcases}
      \zeta(\cdot,t)&\mbox{if}\ t\in B_\zeta\\
      0&\mbox{if}\ t\in[0,T]\setminus B_\zeta.
    \end{dcases}
  \end{equation*}
By Lemma~\ref{2016-01-10:00}, $g$ is well-defined and measurable.
Moreover, $\zeta\in L^{p,q,r}_{\mathcal{P}_T \otimes \mathcal{B}_T}(L_2(U,H))$ implies $g\in L^1([0,T],L^{p,q}_{\mathcal{P}_T}(L_2(U,H))$.
\end{itemize}
Let $X_2$ be the process provided by application of 
the theorem.
Then 
\begin{equation}
  \label{eq:2017-05-07:01}
  X_2(\cdot,t)= \mathfrak I(\zeta(\cdot,t))\ \mbox{in}\ \mathcal{L}^r_{\mathcal{F}_T \otimes \mathcal{B}_T}(\mathbb{T}),\ \mathbb{P}\mbox{-a.e.\ }t\in B_\zeta.
\end{equation}
Define
$$
\Sigma^\zeta_t(\omega)\coloneqq X_2((\omega,t,),t)\qquad \forall(\omega,t)\in\Omega_T,\ t\in [0,T].
$$
Then $\Sigma^\zeta$ is jointly measurable in $(\omega,t)$, and, by \eqref{eq:2017-05-07:01},
 for $m$-a.e.\ $t\in B_\zeta$,
$$
\Sigma^\zeta_t(\omega)=X_2((\omega,t),t)= (\mathfrak I( g(t)))_t(\omega)= (\mathfrak I(\zeta(\cdot,t)))_t(\omega)\qquad \mathbb{P}\mbox{-a.e.\ }\omega\in\Omega.
$$
Moreover, 
\begin{equation}\label{2017-05-07:05}
  \begin{split}
    \int_0^T\mathbb{E} \left[|\Sigma^\zeta_t|^r_F \right]dt 
&=
\int_0^T\mathbb{E} \left[ |X_2((\cdot,t),t)|^r_F\right]dt\\
& \leq
 \int_0^T \mathbb{E} \left[\sup_{s\in[0,T]}|X_2((\cdot,s),t)|_F^r  \right]dt\\
&=
\int_0^T
|X_2(\cdot,t)|^r_{\mathcal{L}^r_{\mathcal{F}_T \otimes \mathcal{B}_T}(\mathbb{T})}dt\\
&=
\mbox{(by \eqref{eq:2017-05-07:01})}=
\int_0^T
| \mathfrak I(\zeta(\cdot,t))|^r_{\mathcal{L}^r_{\mathcal{F}_T \otimes \mathcal{B}_T}(\mathbb{T})}dt\\
&\leq C
\int_0^T
| \zeta(\cdot,t)|^r_{L^{p,q}_{\mathcal{P}_T (L_2(U,H))}}dt
=C|\zeta|_{L^{p,q,r}_{\mathcal{P}_T \otimes \mathcal{B}_T}(L_2(U,H))}^r.
\end{split}
\end{equation}
This shows \eqref{eq:2017-05-07:03}.

Now, if $\Sigma_1$ and $\Sigma_2$ satisfy 
\eqref{eq:2017-05-07:03}
and
\eqref{eq:2017-05-07:04}, with respect to the same $\zeta$, then they belong to the same class in $L^r_{\mathcal{F}_T \otimes \mathcal{B}_T}(F)$, because $m([0,T]\setminus B_\zeta)=0$.
Similarly, if $\zeta_1= \zeta_2$  in $L^{p,q,r}_{\mathcal{P}_T \otimes \mathcal{B}_T}(L_2(U,H))$, then, as noticed in
Remark~\ref{2016-01-09:02},
for $m$-a.e.\ $t\in[0,T]$,
 $ \mathfrak I^{\zeta_1}_t(\omega)= \mathfrak I^{ \zeta_2}_t(\omega)$ $\mathbb{P}$-a.e.\ $\omega\in\Omega$.
Then \eqref{eq:2017-05-07:04} entails $\Sigma^{\zeta_1}=\Sigma^{  \zeta_2}$ for $\mathbb{P} \otimes m$-a.e.\ $(\omega,t)\in \Omega_T$.
This shows that 
 \eqref{2017-05-07:02} is well-defined.
Linearity is clear.
Continuity comes from \eqref{2017-05-07:05}.
\end{proof}

In general, we cannot hope to have versions of $ \mathfrak I^\zeta$ with a better measurability than the one provided by Theorem~\ref{2017-05-07:11}, without further assumptions on $ \mathfrak I$ (observe  that our assumptions on $ \mathfrak  I$ do not take in consideration any progressive measurability
 of the values of $ \mathfrak I$).

\bigskip

We now address the case when 
$
\zeta\in L^{p,q,r}_{\mathcal{P}_T \otimes \mathcal{B}_T}(L_2(U,H))
$  has the form
$$
\zeta((\omega,s),t)=R(t-s)\Phi_s(\omega)\eqqcolon \Phi_R((\omega,s),t)\qquad \forall (\omega,s)\in\Omega_T,\ t\in(s,T],
$$ where $R\colon (0,T]\rightarrow L(H)$ is strongly continuous and $\Phi\in L(U,H)^{\Omega_T}$ is a function.

Under a technical assumption on $R$, we characterize those functions $\Phi\in L(U,H)^{\Omega_T}$ for which 
$\Phi_R$ belongs to $L^0_{\mathcal{P}_T \otimes \mathcal{B}_T}(L_2(U,H))$. This  fact is of interest because it is 
the minimal requirement in order to define the family 
$
  \mathfrak I^{\Phi_R}=\{\mathfrak I^{\Phi_R}_t\}_{t\in B_{\Phi_R}}
$
by
 \eqref{eq:2017-05-07:10} (with $\zeta=\Phi_R$), and
to obtain the joint measurability of $ \mathfrak I^{\Phi_R}$ through Theorem~\ref{2017-05-07:11}.


\begin{assumption}\label{2015-11-06:07}
The function $R\colon (0,T]\rightarrow L(H)$ is strongly continuous and there exists a sequence $\{t_n\}_{n\in \mathbb{N}}\subset (0,T]$ converging to $0$ such that,
if $C\subset H$  is closed, convex, and bounded, 
then $u\in C$
if and only if
$\exists m\in \mathbb{N}\colon R(t_n)u\in R(t_n)C\ \forall n\geq m. $
\end{assumption}

\begin{remark}\label{2016-01-19:00}
 Due to the fact that the closed convex sets in $H$ are the same in the weak and in the strong topology, 
then,
if
the following implication holds for some $\{t_n\}_{n\in \mathbb{N}}\subset (0,T]$ converging to $0$:
  \begin{equation}
    \label{eq:2015-11-07:05}
\{x_n\}_{n\in \mathbb{N}}\subset H \mbox{ bounded such that
$\{R(t_n)x_n\}_{n\in \mathbb{N}}$ is definitely null}
\quad \Longrightarrow\quad
x_n\rightharpoonup 0,
  \end{equation}
Assumption \ref{2015-11-06:07} holds true.
To see it, let ut suppose that
there exists $m\in \mathbb{N}$ such that 
 $R(t_n)u\in R(t_n)C$ for $n\geq m$.
This means that $R(t_n)(u-c_n)=0$ for $n\geq m$.
By \eqref{eq:2015-11-07:05}, $c_n \rightharpoonup u$,
hence $u$ belongs to $C$.

In particular, we notice that
\eqref{eq:2015-11-07:05} 
is satisfied whenever
 $R\colon \mathbb{R}^+\rightarrow L(H)$ is a $C_0$-semigroup on $H$.
In such a case, $R^*$ is a $C_0$-semigroup  (see \cite[pp.\ 43--44, Section 5.14]{Engel2000}, and then we can write, if
 $\{t_n\}_{n\in \mathbb{N}}$ is any bounded sequence converging to $0$ and if $\{x_n\}_{n\in \mathbb{N}}$ 
is such that $\{R(t_n)x_n\}_{n\in \mathbb{N}}$ is definitely null,
$$
\lim_{n\rightarrow \infty}\langle x_n,y\rangle
=
\lim_{n\rightarrow \infty}\langle x_n,R^*(t_n)y\rangle
=
\lim_{n\rightarrow \infty}\langle R(t_n)x_n,y\rangle
=0\qquad \forall y\in H.
$$ 
\end{remark}

\bigskip
In what follows, we denote by $\overline{{\mathcal{P}_T}}$ the completion of ${\mathcal{P}_T}$ with respect to 
$ \mathbb{P} \otimes m$.
If  $\Phi\in L(U,H)^{\Omega_T}$, we denote by $\Phi_R$ the map defined by
\begin{equation}
  \label{eq:2015-11-15:13}
       \Phi_R  \colon    \Omega_T \times [0,T] \rightarrow L(U,H),\ ((\omega,s),t)\mapsto \mathbf{1}_{[0,t)}(s) R(t-s)\Phi_s(\omega).
\end{equation}
     By saying 
that \emph{$\Phi\in L(U,H)^{\Omega_T}$ is strongly measurable}, we mean that 
$$
 (\Omega_T,\mathcal{P}_T)\rightarrow H,\ (\omega,t) \mapsto 
\Phi_t(\omega)u
$$
 is measurable, for all $u\in U$.
Similarly, 
 if $\Phi\in L(U,H)^{\Omega_T}$,
then
\emph{$\Phi_R$ is strongly measurable} if
$\Phi_R(\cdot)u$
is 
 $\mathcal{P}_T \otimes {\mathcal{B}_T}/\mathcal{B}_H$-measurable, for all $u\in U$.

\begin{proposition}\label{2015-11-07:11}
 Let $R\colon (0,T]\rightarrow L(H)$ be strongly continuous and let $\Phi\in L(U,H)^{\Omega_T}$. 
\begin{enumerate}[(i)]
\item\label{2015-11-07:08} If $\Phi$ is strongly measurable, then 
  $ \Phi_R$ is strongly measurable.
\item\label{2015-11-07:07}
Suppose that $R$ satisfies Assumption \ref{2015-11-06:07}.
 If $\Phi_R$ is strongly measurable, then there exists $\hat \Phi\in L(U,H)^{\Omega_T}$ and a $\mathbb{P} \otimes m$-null  set $A\in {\mathcal{P}_T}$  such that $\Phi=\hat \Phi$ on $\Omega_T\setminus A$  and   $\hat \Phi$ is strongly measurable.
\end{enumerate}
\end{proposition}
\begin{proof}
\emph{(\ref{2015-11-07:08})}
Let $\Phi\in L(U,H)^{\Omega_T}$ be strongly measurable.
 Let 
$$
\rho\coloneqq\{0=t_0<\ldots< t_k=T\}\subset [0,T].
$$ 
 Denote 
$\delta(\rho)\coloneqq \sup_{i=0,\ldots,k-1}\{|t_{i+1}-t_{i}|\}$.
Define the function 
$$
\Phi_{R,\rho}\colon  \left(  \Omega_T \times [0,T],\mathcal{P}_T\otimes {\mathcal{B}_T} \right) \rightarrow L(U,H)
$$
by
\begin{equation*}
    \Phi_{R,\rho}((\omega,s),t)\coloneqq \sum_{i=0}^{k-1}\mathbf{1}_{[t_i,t_{i+1})}(t)\mathbf{1}_{[0,t_i)}(s)R(t_i-s)\Phi_{s}(\omega)
+\mathbf{1}_{\{T\}}(t)\mathbf{1}_{[0,T)}(s)R(T-s)\Phi_s(\omega).
\end{equation*}
For all $t\in [0,T]$ and 
$h\in H$, the map
$$
(\Omega_T,{\mathcal{P}_T})\rightarrow H,\ (\omega,s) \mapsto  \mathbf{1}_{[0,t)}(s)R^*(t-s)h
$$
is measurable, by strong continuity of $R$ and Pettis's measurability theorem. Moreover, for $u\in U$,
$$
(\Omega_T,{\mathcal{P}_T})\rightarrow H,\ (\omega,s) \mapsto  \Phi_s(\omega)u
$$
is measurable by assumption, we conclude
that, for $u\in U$ and $t\in[0,T]$,
$$
  \left( \Omega_T,{\mathcal{P}_T} \right) \rightarrow \mathbb{R},\ (\omega,s) \mapsto  \langle\mathbf{1}_{[0,t)}(s)R(t-s)\Phi_s(\omega)u,h\rangle_H
$$
is  measurable. Then, again by Pettis's measurablity theorem,
$$
  \left( \Omega_T,{\mathcal{P}_T} \right) \rightarrow H,\ (\omega,s) \mapsto  \mathbf{1}_{[0,t)}(s)R(t-s)\Phi_s(\omega)u
$$
is measurable, for every $u\in U$ and $t\in[0,T]$. Hence
 $ \Phi_{R,\rho}$ is strongly measurable.
By strong continuity of $R$, 
we have
$$
\lim_{\delta(\rho)\rightarrow 0} \Phi_{R,\rho}((\omega,s),t)u=\Phi_R((\omega,s),t)u\qquad \forall ((\omega,s),t)\in
 \Omega_T\times [0,T],
$$
for every $u\in U$. This shows that $ \Phi_R$ is strongly measurable.

\emph{(\ref{2015-11-07:07})}
Suppose that $\Phi_R$ is strongly measurable.
Let $u\in U$ and let $C\subset H$ be
closed, convex, and bounded. Let $\{t_n\}_{n\in \mathbb{N}}$ be as in 
Assumption \ref{2015-11-06:07}. For $n\in \mathbb{N}$, define
\begin{gather*}
  \Delta_n\coloneqq \{((\omega,s),t)\in \Omega_T\times [0,T]\colon t-s=t_n\}\\[2pt]
B_n\coloneqq
 \left\{ ((\omega,s),t)\in
 \Omega_T\times [0,T]\colon 
\Phi_R((\omega,s),t)u\in
 R(t_n)C
 \right\}\\[2pt]
F_n\coloneqq
 \left\{ (\omega,s)\in \Omega_T\colon 
R(t_n)\Phi_s(\omega)u\in
 R(t_n)C
 \right\}.
\end{gather*}
It is clear that $\Delta_n\in \mathcal{P}_T \otimes {\mathcal{B}_T}$. 
By weak compactness of $C$,
$R(t_n)C$
is closed.
Then, by strong measurability of $\Phi_R$, $B_n\in \mathcal{P}_T \otimes {\mathcal{B}_T}$, hence
 $B_n\medcap \Delta_n\in \mathcal{P}_T \otimes {\mathcal{B}_T}$.

Let $\pi_{\Omega_T}\colon \Omega_T\times [0,T]\rightarrow \Omega_T$ be the projection defined by
$$
\pi_{\Omega_T}((\omega,s),t)\coloneqq (\omega,s).
$$
By the projection theorem (see \cite[p.\ 75, Theorem III-23]{Castaing1977}), $\pi_{\Omega_T}(B_n\medcap \Delta_n)\in \overline{{\mathcal{P}_T}}$.
Notice that
\begin{equation}\label{2015-11-07:09}
  \begin{split}
  \pi_{\Omega_T}(B_n\medcap \Delta_n)&= 
  \left\{ 
    (\omega,s)\in\Omega_T\colon s+t_n\leq T\ \mbox{and}\ R(t_n)\Phi_s(\omega)u\in R(t_n)C \right\}\\
  &=F_n\medcap  \left(  \Omega\times [0,T-t_n] \right).
\end{split}
\end{equation}
By
Assumption \ref{2015-11-06:07}
and by recalling that $\{t_n\}_{n\in \mathbb{N}}\subset (0,T]$ converges to $0$, 
 we have
\begin{equation}
  \label{eq:2015-11-07:10}
  \{(\omega,s)\in\Omega_T\colon   \Phi_s(\omega)u\in C,\ s< T\}=\bigcup_{m\in \mathbb{N}}\bigcap_{n\geq m} ( F_n\medcap  \left(  \Omega\times [0,T-t_n] \right)  ).
\end{equation}
By
\eqref{2015-11-07:09}
and \eqref{eq:2015-11-07:10}, 
we conclude
$
\{(\omega,s)\in\Omega_T\colon \Phi_s(\omega)u\in C,\ s<T\}\in \overline{\mathcal{P}_T}
$.
The slice $\{(\omega,T)\in\Omega_T\colon \Phi_T(\omega)u\in C\}$ is a $\mathbb{P} \otimes m$-null set.
Then 
$$
\{(\omega,s)\in\Omega_T\colon \Phi_s(\omega)u\in C\}\in \overline{\mathcal{P}_T}.
$$
Since this holds for every closed, convex, bounded set $C$, hence for balls, and since $H$ is separable, we have that  $\Phi u$ is $\overline{{\mathcal{P}_T}}/\mathcal{B}_H$-measurable,
for every $u\in U$. 

Now let $\{u_n\}_{n\in \mathbb{N}}$  be a dense subset of $U$.
Since $\overline{{\mathcal{P}_T}}$ is the completion of ${\mathcal{P}_T}$ with respect to $ \mathbb{P} \otimes  m$, and since $H$ is separable,
 for every $n\in \mathbb{N}$ there exists $A_n\in {\mathcal{P}_T}$ such that $ \mathbb{P} \otimes m(A_n)=0$ and  $\mathbf{1}_{A_n}\Phi u_n$ is ${\mathcal{P}_T}/\mathcal{B}_H$-measurable. Let $A\coloneqq \medcup_{n\in \mathbb{N}}A_n$. Then $A\in {\mathcal{P}_T}$, $ \mathbb{P}\otimes m(A)=0$, and $\mathbf{1}_A\Phi u_n$ is ${\mathcal{P}_T}/\mathcal{B}_H$-measurable for every $n\in \mathbb{N}$. Since $\Phi_s(\omega)\in L(U,H)$ for every $(\omega,s)\in\Omega_T$, by density of $\{u_n\}_{n\in \mathbb{N}}$ we conclude that $\mathbf{1}_A \Phi u$ is ${\mathcal{P}_T}/\mathcal{B}_H$-measurable for every $u\in U$. This concludes the proof of
\emph{(\ref{2015-11-07:07})} and of the proposition.
\end{proof}

\noindent  We will make use of the following lemma, whose proof can be found in \cite[Ch.\ 1]{Fabbri}.
\begin{lemma}\label{2015-11-08:02}
Let $(G,\mathcal{G})$ be a measurable space.   
Let
$f\colon (G,\mathcal{G})\rightarrow L_2(U,H)$.
Then
 $f(\cdot)u$ is $\mathcal{G}/\mathcal{B}_H$-measurable, for all $u\in U$, if and only if
$f$ 
is $\mathcal{G}/\mathcal{B}_{L_2(U,H)}$-measurable.
\end{lemma}



%

Under  Assumption~\ref{2015-11-06:07},
the following theorem characterizes those functions $\Phi\in L(U,H)^{\Omega_T}$ for which 
$\Phi_R$ belongs to $L^0_{\mathcal{P}_T \otimes \mathcal{B}_T}(L_2(U,H))$.

\begin{theorem}\label{2017-05-07:12}
 Let $R\colon (0,T]\rightarrow L(H)$ be strongly continuous and let $\Phi\in L(U,H)^{\Omega_T}$. 
\begin{enumerate}[(i)]
\item\label{2015-11-07:13} If $\Phi$ is strongly measurable
and if $\mathbf{1}_{[0,t)}(s)R(t-s)\Phi_s(\omega)\in L_2(U,H)$ for all $((\omega,s),t)\in  \Omega_T\times [0,T]$, then
  $ \Phi_R$ is measurable as an $L_2(U,H)$-valued map (that is when $L_2(U,H)$ is endowed with its Borel $\sigma$-algebra).
\item\label{2015-11-07:12}
Suppose that $R$ satisfies Assumption \ref{2015-11-06:07}.
 If $\Phi_R$ has values in $L_2(U,H)$ and if
it
is measurable as an $L_2(U,H)$-valued map, then there exists $\hat \Phi\in L(U,H)^{\Omega_T}$ and a $\mathbb{P} \otimes m$-null  set $A\in {\mathcal{P}_T}$  such that $\Phi=\hat \Phi$ on $\Omega_T\setminus A$,   $\hat \Phi$ is strongly measurable, and  $\mathbf{1}_{[0,t)}(s)R(t-s)\hat \Phi_s(\omega)\in L_2(U,H)$ for all  $((\omega,s),t)\in \Omega_T\times [0,T]$.
\end{enumerate}
\end{theorem}
\begin{proof}
Apply
Proposition~ \ref{2015-11-07:11}
and  Lemma~\ref{2015-11-08:02}.
\end{proof}







\begin{example}\label{2016-01-21:00}
Let $Q$ be a positive self-adjoint operator
of trace class in $H$ and let $W$ be a $U$-valued $Q$-Wiener process with respect to $(\Omega,\mathcal{F},\mathbb{F},\mathbb{P})$.
Let $U_0\coloneqq Q^{1/2}(U)$ be the Hilbert space isometric to $U$ through $Q^{-1/2}\colon U_0\rightarrow U$.
By
\cite[p.\ 114, Theorem 4.37]{DaPrato2014}, for $p\geq 2$,
 the stochastic integral is a linear and continuous map
$$
 \mathfrak  I_W\colon L^{p,2}_{\mathcal{P}_T}(L_2(U_0,H))\rightarrow
\mathcal{L}^p_{\mathcal{P}_T}(\Cb{H}),\ \Psi \mapsto \Psi\cdot W\coloneqq \int_0^\cdot \Psi_sdW_s.
$$
Let $R$ be as in
Assumption~\ref{2015-11-06:07}.
Let $\Phi\in L(U_0,H)^{\Omega_T}$ be strongly measurable and such that $R(t-s)\Phi_s(\omega)\in L_2(U_0,H)$ for
$(\omega,s)\in\Omega_T$, $t\in (s,T]$.
Then, by Theorem~\ref{2017-05-07:12}\emph{(\ref{2015-11-07:13})}, $\Phi_R\in L^0_{\mathcal{P}_T \otimes \mathcal{B}_T}(L_2(U_0,H))$.
 If $|\Phi_R|_{p,2,p}<\infty$, then we can apply 
Theorem~\ref{2017-05-07:11}, according to which the process
$$
 \left\{ \int_0^t R(t-s)\Phi_sdW_s  \right\}_{t},
$$
which is well-defined for a.e.\ $t\in[0,T]$, has an $ \mathcal{F}_T \otimes \mathcal{B}_T$-jointly measurable version.
\end{example}


\subsection{Continuous  version}
\label{2016-01-22:02}

In this section we 
review the factorization method used to show
existence of continuous version
 of stochastic  convolutions made with respect to a $C_0$-semigroup.

 \begin{notation}
   Throughout this section
\begin{itemize}
\itemsep=-1mm
\item $S$ denotes a strongly continuous semigroup on $H$ and $M\coloneqq \sup_{t\in[0,T]}|S_t|_{L(H)}$;
\item $\mathbb{W}\coloneqq \Cb{H}$;
\item for $\beta\in (0,1)$, $c_\beta$ denotes the number
$
c_\beta\coloneqq  \left( \int_0^1(1-w)^{\beta-1}w^{-\beta}dw  \right) ^{-1}$.
\end{itemize}
\end{notation}
\noindent As noticed in Remark \ref{2016-01-19:00}, $S$ verifies
 Assumption 
\ref{2015-11-06:07}.

\medskip
The factorization method relies on the semigroup property of $S$ and on the fact that continuous linear operator commutes with stochastic integral.
We  rephrase this commutativity assumption in our setting through the following


\begin{assumption}\label{2016-01-15:01}
 Let $p,q,r\in [1,\infty)$, and let
$$
   \mathfrak I\colon  L^{p,q}_{{\mathcal{P}_T}}  \left(  L_2(U,H) \right)  \mapsto \mathcal{L}^ r_{\mathcal{F}_T \otimes \mathcal{B}_T}(\mathbb{W})
$$
 be a linear and continuous
  operator such that
  \begin{equation}\label{2016-01-17:04}
    Q  \left( \mathfrak  I\Phi \right) = \mathfrak I (  Q\Phi )\ \mbox{in }\mathcal{L}^{r}_{\mathcal{F}_T \otimes \mathcal{B}_T}(\mathbb{W})\
(\footnote{$Q$ applied to a process $\Phi$ means the pointwise composition $Q(\Phi_t(\omega))$, for $(\omega,t)\in\Omega_T$.})
\qquad
\forall Q\in L(H).
  \end{equation}
\end{assumption}


\bigskip
\noindent For $p,q,r\in[1,\infty)$ and $\beta\in [0,1)$, $\Lambda^{p,q,r}_{\mathcal{P}_T,S,\beta}(L(U,H))$ denotes the vector space of equivalence classes of strongly measurable functions 
$\Phi\in L(U,H)^{\Omega_T}$ 
such that
\begin{equation}\label{2017-05-08:02}
\left( 
\int_0^T \left( \int_0^t
(t-s)^{-\beta q}
 \left(  \mathbb{E} \left[ |S(t-s)\Phi_s|_{L_2(U,H)}^p \right]  \right) ^{q/p}ds \right)^{r/q} dt
 \right) ^{1/r}<\infty.
\end{equation}
Two functions $\Phi_1$, $\Phi_2$, are in the same class if the quantity \eqref{2017-05-08:02} is $0$ for $\Phi=\Phi_1-\Phi_2$.
This implies, for all $u\in U$, for
$\mathbb{P} \otimes m$-a.e.\ $(\omega,s)\in\Omega_T$,
$$
\mathbf{1}_{[0,t)}(s)S(t-s)(\Phi_1)_s u=
\mathbf{1}_{[0,t)}(s)S(t-s)(\Phi_2)_s u\qquad 
 m\mbox{-a.e.\ }t\in[0,T],
$$
hence, by strong continuity of $S$, for all $u\in U$,
$$
(\Phi_1)_s u=
(\Phi_2)_s u\qquad 
\mathbb{P} \otimes m\mbox{-a.e.\ }(\omega,s)\in\Omega_T.
$$
By separability of $U$ we conclude
that $\Phi_1=\Phi_2$ in $\Lambda^{p,q,r}_{\mathcal{P}_T,S,\beta}(L(U,H))$ if and only if
 $(\Phi_1)_s(\omega)=(\Phi_2)_s(\omega)$ in $L(U,H)$ $\mathbb{P} \otimes m$-a.e.\ $(\omega,s)\in\Omega_T$.

\bigskip
For $\Phi\in \Lambda^{p,q,r}_{\mathcal{P}_T,S,\beta}(L(U,H))$, we define, for all 
$(\omega,s)\in\Omega_T$ and $ t\in[0,T]$,
\begin{gather*}
\label{2017-05-08:12}
\Phi_{S,\beta}((\omega,s),t)\coloneqq \mathbf{1}_{[0,t)}(s)(t-s)^{-\beta}S(t-s)\Phi_s(\omega)\\[5pt]
\Phi_{S}((\omega,s),t)\coloneqq \mathbf{1}_{[0,t)}(s)S(t-s)\Phi_s(\omega),
\end{gather*}
By Theorem~\ref{2017-05-07:12}\emph{(\ref{2015-11-07:13})},
$\Phi_{S,\beta}\in L^0_{\mathcal{P}_T \otimes \mathcal{B}_T}(L_2(U,H))$, and 
\eqref{2017-05-08:02} can be written as 
\begin{equation}
  \label{eq:2017-05-24:01}
  |\Phi_{S,\beta}|_{p,q,r}< \infty.
\end{equation}
Then, through the well-defined map 
$$
\Lambda^{p,q,r}_{\mathcal{P}_T,S,\beta}(L(U,H))\rightarrow L^{p,q,r}_{\mathcal{P}_T \otimes \mathcal{B}_T}(L_2(U,H)),\ 
\Phi \mapsto \Phi_{S,\beta},
$$
 $\Lambda^{p,q,r}_{\mathcal{P}_T,S,\beta}(L(U,H))$ is identified  with a subspace of $L^{p,q,r}_{\mathcal{P}_T \otimes \mathcal{B}_T}(L_2(U,H))$.
In particular, the map
$$
\Lambda^{p,q}_{\mathcal{P}_T,S,\beta}(L(U,H))\rightarrow \mathbb{R}^+,\ \Phi \mapsto |\Phi_{S,\beta}|_{p,q,r}
$$
is a norm.
In what follows we always consider 
$\Lambda^{p,q}_{\mathcal{P}_T,S,\beta}(L(U,H))$  endowed with the norm $|\#_{S,\beta}|_{p,q,r}$.

\bigskip
Again by Theorem~\ref{2017-05-07:12}\emph{(\ref{2015-11-07:13})},
$\Phi_S\in L^0_{\mathcal{P}_T \otimes \mathcal{B}_T}(L_2(U,H))$. 
Moreover, for all $t'\in[0,T]$, we have, by 
applying Minkowski's inequality for integrals 
(see \cite[p.\ 194, 6.19]{Folland1999}),
\begin{equation*}
  \begin{split}
    |\Phi_S(\cdot,t)|_{p,q}
    &=
c_\beta 
 \left(    \int_0^{t'} \left(
 \int_s^{t'}(t'-t)^{\beta-1}(t-s)^{-\beta} 
 \left( \mathbb{E} \left[ 
       |S(t'-s)\Phi_s|_{L_2(U,H)}^p \right]  \right) ^{1/p}dt
 \right) ^q
ds \right) ^{1/q}\\
&\leq c_\beta \int_0^{t'}
(t'-t)^{\beta-1}
 \left( 
\int_0^t
(t-s)^{-\beta q}
 \left( 
\mathbb{E} \left[ |S(t'-s)\Phi_s|_{L_2(U,H)}^p \right] 
 \right) ^{q/p}
 ds
\right)^{1/q}dt.
  \end{split}
\end{equation*}
Now, if we take $r>1$ and $\beta\in (1/r,1)$,  by
applying  H\"older's inequality to the last term and  writing $S(t'-s)=S(t'-t)S(t-s)$,
\begin{equation}
  \label{eq:2017-05-09:01}
    |\Phi_S(\cdot,t')|_{p,q}\leq
 c_\beta M  \left( \int_0^{T}w^{\frac{(\beta-1)r}{r-1}} dw\right) ^{(r-1)/r}
|\Phi_{S,\beta}|_{p,q,r}< \infty.
\end{equation}
This shows that
\begin{equation}
  \label{eq:2017-05-09:02}
  \Phi_S(\cdot,t')\in L^{p,q}_{\mathcal{P}_T}(L_2(U,H)),
\  \forall t'\in[0,T]. 
\end{equation}

\begin{theorem}\label{2016-01-17:06}
Let $p,q\in[1,\infty)$, $r\in (1,\infty)$,
 $\beta\in(1/r,1)$.
Let $ \mathfrak I$ be as in 
Assumption~\ref{2016-01-15:01}.
Then there exists a unique linear and continuous function
\begin{equation}
  \label{eq:2017-05-08:20}
  \mathbf{C}\colon
 \Lambda^{p,q,r}_{\mathcal{P}_T,S,\beta}(L(U,H))\rightarrow
\mathcal{L}^r_{\mathcal{F}_T \otimes \mathcal{B}_T}(
\mathbb{W})
\end{equation}
such that,
for all $\Phi\in \Lambda^{p,q,r}_{\mathcal{P}_T,S,\beta}(L(U,H))$,
for all $t\in[0,T]$,
\begin{equation}
  \label{2017-05-08:01}
   \left( \mathfrak I \left( \mathbf{1}_{[0,t)}(\cdot)S(t-\cdot)\Phi \right)  \right) _t= \left( \mathbf{C}(\Phi) \right) _t\qquad \mathbb{P}\mbox{-a.e..}
\end{equation}
The operator norm of $\mathbf{C}$ is bounded by a constant 
depending only on $\beta$, $r$, $T$, $M$, and on the operator norm of $ \mathfrak I$.
\end{theorem}
\begin{proof}
Let $\Phi\in \Lambda^{p,q,r}_{\mathcal{P}_T,S,\beta}(L(U,H))$.
First notice that the left-hand side of
\eqref{2017-05-08:01}
is meaningful because of \eqref{eq:2017-05-09:02}.
We now construct $\textbf{C}(\Phi)$.
Fix $t'\in[0,T]$, and define
$$
\Phi^{(t')}_{S,\beta}((\omega,s),t)\coloneqq c_\beta \mathbf{1}_{[0,t')}(t)(t'-t)^{\beta-1}
\mathbf{1}_{[0,t)}(s)(t-s)^{-\beta}
S(t'-s)\Phi_s(\omega)\qquad (\omega,s)\in\Omega_T,\ t\in[0,T].
$$
By Theorem~\ref{2017-05-07:12}\emph{(\ref{2015-11-07:13})},
 $\Phi^{(t')}_{S,\beta}\in L^0_{\mathcal{P}_T \otimes \mathcal{B}_T}(L_2(U,H))$.
Moreover, 
  \begin{equation*}
    \begin{split}
|\Phi^{(t')}_{S,\beta}|_{p,q,1}
&=c_\beta\int_0^{t'} (t'-t)^{\beta-1}\left( \int_0^t (t-s)^{-q\beta }   \left( \mathbb{E} \left[ |S(t'-s)\Phi_s|_{L_2(U,H)}^p \right] \right) ^{q/p}  ds\right) ^{1/q}dt\\
&\leq c_\beta M
\left( \int_0^{T}w^{\frac{(\beta-1)r}{r-1}} dw\right) ^{(r-1)/r}
|\Phi_{S,\beta}|_{p,q,r}<\infty.
    \end{split}
  \end{equation*}
Then $\Phi^{(t')}_{S,\beta}\in L^{p,q,1}_{\mathcal{P}_T \otimes \mathcal{B}_T}(L_2(U,H))$.
By Lemma~\ref{2016-01-10:00},  
the map
\begin{equation}
  \label{eq:2017-05-11:01}
  g\colon B_0\rightarrow L^{p,q}_{\mathcal{P}_T}(L_2(U,H)),\ t \mapsto \Phi^{(t')}_{S,\beta}(\cdot,t),
\end{equation}
where  $B_0$ is the set of $t$ such that  $|\Phi^{(t')}_{S,\beta}(\cdot,t)|_{p,q}<\infty$, is Borel measurable.
Let us define $g=0$ on $[0,T]\setminus B_0$.
By $\Phi^{(t')}_{S,\beta}\in L^{p,q,1}_{\mathcal{P}_T \otimes \mathcal{B}_T}(L_2(U,H))$
and by measurability of \eqref{eq:2017-05-11:01},
 we have $g\in L^1([0,T],L^{p,q}_{\mathcal{P}_T}(L_2(U,H))$.
We can then apply Theorem~\ref{2015-09-21:01}, 
with the following data:
\begin{itemize}
\itemsep=-1mm
\item $G=[0,T]$, $\mathcal{G}=\mathcal{B}_T$, $\mu=m$;
\item $D_1=\Omega$, $D_2=[0,T]$, $D=\Omega_T$, $\mathcal{D}=\mathcal{P}_T$, $\nu_1=\mathbb{P}$, $\nu_2=m$;
\item $E=L_2(U,H)$;
\item $F=H$;
\item  $\mathfrak L= \mathfrak I$;
\item $g$ as above.
\end{itemize}
The theorem provides measurable functions
$$
X_1\colon (\Omega_T\times [0,T],\mathcal{P}_T \otimes \mathcal{B}_T)\rightarrow  L_2(U,H)
\qquad
X_2\colon (\Omega_T\times [0,T],(\mathcal{F}_T \otimes \mathcal{B}_T) \otimes \mathcal{B}_T)\rightarrow H
$$
such that,
 for $m$-a.e.\ $t\in[0,T]$,
\begin{equation}
  \label{eq:2017-05-08:05}
  \begin{dcases}
X_1(\cdot,t)=g(t)\ \mbox{in\ }L^{p,q}_{\mathcal{P}_T}(L_2(U,H))\\
X_2(\cdot,t)= \mathfrak I (g(t))\ \mbox{in\ }\mathcal{L}^r_{\mathcal{F}_T \otimes \mathcal{B}_T}(\mathbb{W}),
\end{dcases}
\end{equation}
and
\begin{equation}
  \label{eq:2017-05-08:16}
\mbox{for\ }\mathbb{P}\mbox{-a.e.\ }
\omega\in\Omega,\   ( \mathfrak I Y)_t(\omega)=\int_0^TX_2((\omega,t),s)ds\ \forall t\in[0,T],
\end{equation}
where
\begin{equation}
  \label{eq:2017-05-08:06}
  Y_t(\omega)=\int_0^TX_1((\omega,t),s)ds,\ \forall (\omega,t)\in\Omega_T.
\end{equation}
By \eqref{eq:2017-05-08:05}, by definition of $g$, and by joint measurability of $X_1$ and $\Phi_{S,\beta}^{(t')}$, we have
\begin{equation}
  \label{2017-05-08:08}
  X_1((\omega,s),t)=\Phi^{(t')}_{S,\beta}((\omega,s),t)\qquad (\mathbb{P} \otimes m )  \otimes m\mbox{-a.e.\ }((\omega,s),t)\in\Omega_T\times [0,T].
\end{equation}
Then
\eqref{eq:2017-05-08:06} becomes
\begin{equation}
  \label{eq:2017-05-08:07}
  Y_t(\omega)=\int_0^T\Phi_{S,\beta}^{(t')}((\omega,t),s)ds=\mathbf{1}_{[0,t')}(t)S(t'-t)\Phi_t(\omega)\qquad \mathbb{P} \otimes m\mbox{-a.e.\ }(\omega,t)\in\Omega_T,
\end{equation}
hence, in particular,
\begin{equation}
  \label{eq:2017-05-08:14B}
  ( \mathfrak I Y)_{t'}(\omega) =\mathfrak I \left( \mathbf{1}_{[0,t')}S(t'-\cdot)\Phi \right)_{t'}(\omega)\qquad \mathbb{P}\mbox{-a.e.}\ \omega\in\Omega.
\end{equation}

On the other hand, for $m$-a.e.\ $t\in[0,T]$, 
\begin{equation*}
  g(t)=c_\beta\mathbf{1}_{[0,t')}(t)(t'-t)^{\beta-1}S(t'-t)\Phi_{S,\beta}(\cdot,t)\mbox{\ in\ }L^{p,q}_{\mathcal{P}_T}(L_2(U,H)).
\end{equation*}
Then, by assumption on $ \mathfrak I$, we have, for $m$-a.e.\ $t\in[0,T]$,
\begin{equation}
   \mathfrak I (g(t))=
c_\beta\mathbf{1}_{[0,t')}(t)(t'-t)^{\beta-1}S(t'-t)
 \mathfrak I \left( 
\Phi_{S,\beta}(\cdot,t) \right) \mbox{\ in\ }
\mathcal{L}^r_{ \mathcal{F}_t \otimes \mathcal{B}_T}(\mathbb{W}),
\end{equation}
hence, in particular,
for $m$-a.e.\ $t\in[0,T]$,
\begin{equation}
  \label{eq:2017-05-08:10}
       (\mathfrak I (g(t)))_t(\omega)
=
c_\beta\mathbf{1}_{[0,t')}(t)(t'-t)^{\beta-1}S(t'-t)
  \left( \mathfrak I \left( 
\Phi_{S,\beta}(\cdot,t) \right) \right) _t(\omega)\quad 
\mathbb{P}\mbox{-a.e.\ }\omega\in\Omega.
\end{equation}
Now our aim is to replace the last factor in
\eqref{eq:2017-05-08:10} with a process jointly measurable in $(\omega,t)$.
We noticed in \eqref{eq:2017-05-24:01}
that
$
\Phi_{S,\beta}\in L^{p,q,r}_{\mathcal{P}_T \otimes \mathcal{B}_T}(L_2(U,H))
$.
We can then apply Theorem~\ref{2017-05-07:11}.
Let 
\begin{equation}
  \label{eq:2017-05-08:18}
  \Sigma^{\Phi_{S,\beta}}\coloneqq \mathbf{J}(\Phi_{S,\beta})\in L^r_{\mathcal{F}_T \otimes \mathcal{B}_T}(H)
\end{equation}
be the process obtained by applying the map \eqref{2017-05-07:02} to $\Phi_{S,\beta}$.
We know that $\Sigma^{\Phi_{S,\beta}}_t(\omega)$ is 
 $\mathcal{F}_T \otimes \mathcal{B}_T$-measurable in $(\omega,t)\in \Omega_T$ and that, for $m$-a.e.\ $t\in[0,T]$,
\begin{equation}
  \label{eq:2017-05-08:03}
   \left(  \mathfrak I (\Phi_{S,\beta}(\cdot,t)) \right) _t(\omega)=\Sigma^{\Phi_{S,\beta}}_t(\omega)\qquad \mathbb{P}\mbox{-a.e.\ }\omega\in\Omega.
\end{equation}
By \eqref{eq:2017-05-08:10} and \eqref{eq:2017-05-08:03}, we can write, for $m$-a.e.\ $t\in[0,T]$, 
\begin{equation}
  \label{2017-05-08:11}
   \left(  \mathfrak I (g(t)) \right) _t(\omega)=
c_\beta\mathbf{1}_{[0,t')}(t)(t'-t)^{\beta-1}S(t'-t)
\Sigma^{\Phi_{S,\beta}}_t(\omega)\qquad
\mathbb{P}\mbox{-a.e.\ }\omega\in\Omega.
\end{equation}
Then, by \eqref{eq:2017-05-08:05} and taking into account  the joint measurability of $X_2$ and  $\Sigma^{\Phi_{S,\beta}}$,
\begin{equation}
  \label{eq:2017-05-08:13}
  X_2((\omega,t'),t)=
c_\beta\mathbf{1}_{[0,t')}(t)(t'-t)^{\beta-1}S(t'-t)
\Sigma^{\Phi_{S,\beta}}_t(\omega)
\qquad
 \mathbb{P} \otimes  m\mbox{-a.e.\ }(\omega,t)\in \Omega_T.
\end{equation}
By \eqref{eq:2017-05-08:16},
\eqref{eq:2017-05-08:13},
\eqref{eq:2017-05-08:14B},
 we finally obtain, for $\mathbb{P}$-a.e.\ $\omega\in\Omega$,
\begin{equation}\label{2017-05-09:03}
  \begin{split}
\mathfrak I \left(\mathbf{1}_{[0,t')} S(t'-\cdot)\Phi \right)_{t'}(\omega)&=      ( \mathfrak I Y)_{t'}(\omega) =
\int_0^TX_2((\omega,t'),s)ds\\
&=
c_\beta\int_0^{t'} (t'-t)^{\beta-1}S(t'-t)\Sigma^{\Phi_{S,\beta}}_t(\omega)dt.
  \end{split}
\end{equation}
Now define the process $\textbf{C}(\Phi)$  by
\begin{equation}
  \label{eq:2015-11-04:10B}
  (\mathbf{C}(\Phi))_t(\omega)\coloneqq
 \begin{dcases}
 c_\beta \int_0^t(t-s)^{\beta-1}S(t-s)\Sigma_s^{\Phi_{S,\beta}}(\omega)ds&\textrm{if}\ \Sigma^{\Phi_{S,\beta}}(\omega)\in L^r([0,T],H)\\
 0&\textrm{otherwise}
\end{dcases}
\end{equation}
for all $(\omega,t)\in \Omega_T$.
By
\cite[p.\ 129, Proposition 5.9]{DaPrato2014},
 $\mathbf{C}(\Phi)$ is well-defined and pathwise continuous.
By H\"older's inequality,
\begin{equation}
  \label{eq:2015-09-21:09}
  |\mathbf{C}(\Phi)(\omega)|_{\infty}\leq C_{\beta,
r,T,M}|\Sigma^{\Phi_{S,\beta}}(\omega)|_{L^r([0,T],H)}\qquad\forall \omega\in \Omega,
\end{equation}
where $ C_{\beta,r,T,M}$ depends only on $\beta,r,T,M$.
Hence, by recalling \eqref{eq:2017-05-08:18},
\begin{equation}
  \label{eq:2017-05-11:02}
  |\mathbf{C}(\Phi)|_{\mathcal{L}^r_{\mathcal{F}_T \otimes \mathcal{B}_T}(\mathbb{W})}\leq 
C_{\beta,r,T,M} |\Sigma^{\Phi_{S,\beta}}|_{L^r_{\mathcal{F}_T \otimes \mathcal{B}_T}(H)}
\leq C_{\beta,r,T,M, |\mathfrak I|}|\Phi_{S,\beta}|_{p,q,r}.
\end{equation}
where $ C_{\beta,r,T,M, | \mathfrak I|}$ depends only on $\beta,r,T,M$, and on the operator norm $| \mathfrak I|$ of  $ \mathfrak I$.
Moreover, since $t'\in[0,T]$ was arbitrary chosen, and since the choice of $\Sigma^{\Phi_{S,\beta}}$ does not depend on $t'$,
 \eqref{2017-05-09:03} gives,  for all $t'\in[0,T]$,
\begin{equation}
  \label{eq:2017-05-08:19}
    \mathfrak I \left( \mathbf{1}_{[0,t')}S(t'-\cdot)\Phi \right)_{t'}=
(\mathbf{C}(\Phi))_{t'}\qquad \mathbb{P}\mbox{-a.e.}.
\end{equation}
It is clear that
the process $\mathbf{C}(\Phi)$ is uniquely identified 
by \eqref{eq:2017-05-08:19} in $\mathcal{L}^r_{\mathcal{F}_T \otimes \mathcal{B}_T}(\mathbb{W})$, because it is continuous.
Moreover, if $\Phi=\Phi'$ in $\Lambda^{p,q,r}_{\mathcal{P}_T,S,\beta}(L(U,H))$, then
$\mathbf{C}(\Phi)=\mathbf{C}(\Phi')$ in 
$\mathcal{L}^r_{\mathcal{F}_T \otimes \mathcal{B}_T}(\mathbb{W})$.
Linearity of $\textbf{C}$ is clear as well.
This concludes the proof that the map \eqref{eq:2017-05-08:20} is well-defined on 
$\Lambda^{p,q,r}_{\mathcal{P}_T,S,\beta}(L(U,H))$, linear, 
and that \eqref{2017-05-08:01} is satisfied.
Continuity with operator norm bounded by a constant depending only on $\beta$, $r$, $T$, $M$, $| \mathfrak I|$
is due to \eqref{eq:2017-05-11:02}.
\end{proof}

We remark that the joint measurability of $X_1$, $X_2$, $\Sigma^{\Phi_{S,\beta}}$, provided by 
Theorem~\ref{2015-09-21:01} and Theorem~\ref{2017-05-07:11}, play a central role in order to obtain the factorization formula \eqref{eq:2015-11-04:10B}.

\addcontentsline{toc}{chapter}{References}
\bibliographystyle{plain}
\bibliography{rosestolato_2018-06-20_note-stoch-conv.bbl}

\end{document}